\documentclass[11pt]{article}
  \usepackage{amsmath}
  \usepackage{amsthm}
  \usepackage{latexsym}
  \usepackage{amscd}
  \usepackage{amssymb}
  \usepackage{verbatim}
  \usepackage{authblk}
  \usepackage{bm}
\usepackage[left=3cm,top=3cm,right=3cm,bottom=3cm]{geometry}
\newcommand\val[1]{{\lbrack\!\lbrack} {#1}{\rbrack\!\rbrack}}
\newcommand{\commment}[1]{}
\newcommand{\Prop}{\mathsf{Prop}}
\newcommand{\f}{\mathcal{F}}
\newcommand{\Ll}{\mathrm{ML}}
 \newcommand{\pp}{\mathcal{P}}
 \newcommand{\p}{\mathcal{P}}
  \newcommand{\s}{\mathcal{S}}
   
  \newcommand{\x}{\mathcal{X}}
   \newcommand{\y}{\mathcal{Y}}
\newcommand{\Diamondblack}{\blacklozenge}
\newcommand{\nomj}{\mathbf{j}}
\newcommand{\nomi}{\mathbf{i}}
\newcommand{\cnomm}{\mathbf{m}}
\newcommand{\cnomn}{\mathbf{n}}

\newcommand{\ML}{\ensuremath{\mathrm{ML}}}
\renewcommand{\phi}{\varphi}
\renewcommand{\emptyset}{\varnothing}

\newcommand{\marginnote}[1]{\marginpar{\raggedright\tiny{#1}}}

 \newtheorem{theorem}{Theorem}[section]
 \newtheorem{lemma}[theorem]{Lemma}
 \newtheorem{prop}[theorem]{Proposition}
 \newtheorem{fact}[theorem]{Fact}
 \newtheorem{cor}[theorem]{Corollary}

 \theoremstyle{definition}
 \newtheorem{definition}[theorem]{Definition}
 \newtheorem{example}[theorem]{Example}

 \newtheorem{remark}[theorem]{Remark}

 \theoremstyle{remark}

\begin{document}
 \title{Algebraic modal correspondence: Sahlqvist and beyond
 }
\author[1]{Willem Conradie\thanks{The research of the first author was supported by grant number NRF UID 70554 of the National Research Foundation of South Africa.}}
\author[1,2]{Alessandra Palmigiano \thanks{The research of the second author has been supported by the the NWO Vidi grant 016.138.314, by the NWO Aspasia grant 015.008.054, and by a Delft Technology Fellowship awarded in 2013.}}
\author[3]{Sumit Sourabh}
\affil[1]{University of Johannesburg}
\affil[2]{Delft University of Technology}
\affil[3]{University of Amsterdam}


%
\maketitle


\begin{abstract}
The present paper proposes a new introductory treatment of the very well known Sahlqvist correspondence theory for classical modal logic. The first motivation for the present treatment is {\em pedagogical}: classical Sahlqvist correspondence is presented  in a uniform and modular way, and, unlike the existing textbook accounts,  extends itself to a class of formulas laying outside the Sahlqvist  class proper. The second motivation is {\em methodological}: the present treatment aims at highlighting  the {\em algebraic} and {\em order-theoretic} nature of the correspondence mechanism. The exposition remains elementary and does not presuppose any  previous knowledge or familiarity with the algebraic approach to logic. However, it provides the underlying motivation and basic intuitions for the recent developments in the Sahlqvist theory of nonclassical logics, which compose the so-called unified correspondence theory.  
\end{abstract}

\tableofcontents

\section*{Introduction}
Modal logics are perhaps the best known logics after classical propositional and predicate logic. In their modern form they were introduced in the 1930s, as enriched formal languages in which one can express and reason about \emph{modes of truth}, e.g., the {\em possible}, {\em necessary}, {\em usual} or {\em past} truth of propositions. Syntactically, the language of modal logic is an expansion of classical propositional language with new connectives, so as to have formulas such as $\Box \phi$ or $\Diamond\phi$, the intended meaning of which is `$\phi$ is necessary/obligatory/always true in the past$\dots$' and `$\phi$ is possible/permitted/sometimes true in the past$\dots$', respectively. Modal logics are widely applied in fields as diverse as program verification in theoretical computer science \cite{GrVe08}, natural language semantics in formal philosophy \cite{vBM1997}, multi-agent systems in AI \cite{Gabbay:1993}, 
foundations of arithmetics \cite{ArtBek05}, game theory in economics \cite{ParPau03}, categorization theory in social science  and management science \cite{PoHa10}.

This success is due to the peculiar but natural way modal logic is interpreted in relational structures (Kripke frames), paired with the ubiquity of these structures in science and philosophy. But the existence of this interpretation on its own would not be enough: correspondence theory is the mathematical tool which makes it possible to exploit the relational interpretation of modal logic by establishing systematic connections with the languages of first- and second-order logic which are naturally interpreted on relational structures.
In particular, Sahlqvist correspondence theory
%
provides a syntactic characterization of a class of formulas (the so-called {\em Sahlqvist formulas}) which are equivalent to first-order conditions on  Kripke frames which are effectively computable from the given modal formula. In a sense, Sahlqvist correspondence theory works as a meta-semantic tool which makes it possible to understand the `meaning' of  modal axioms (which, interestingly, in the best known cases arise from non-mathematical, e.g.\ philosophical considerations) in terms of the condition expressed by its first-order frame correspondent. In this way, for instance, $\Box p\rightarrow p$ can be understood as the `reflexivity axiom', $\Diamond\Diamond p\rightarrow \Diamond p$ as the `transitivity axiom', $\Diamond \Box p\rightarrow \Box \Diamond p$ as the `confluence axiom', and so on. 

 Sahlqvist theory is currently a very active field of research. This field has significantly broadened its scope in recent years, extending the state of the art and the benefits of Sahlqvist theory from modal logic to a class of logics which includes, among others, intuitionistic and lattice-based (modal) logics \cite{GNV, CoPa10}, substructural logics \cite{Kurtonina,Suz11,CoPa11}, non-normal modal logics \cite{FrPaSa14,PaSoZh14a}, hybrid logics \cite{ConRob}, many-valued logics \cite{manyval},  mu-calculus \cite{BeHovB12,CFPS,CoCr14,BeSo,CCPZ}, and coalgebraic logic \cite{LiPaSaSc12,LiPaSc}. 
The common ground to these results is the recognition that algebraic and order-theoretic notions play a fundamental role in the various incarnations of the Sahlqvist phenomenon. This recognition initially concerned only Sahlqvist {\em canonicity} results \cite{Jo94, GhMe97}. More recently, \cite{CoPa10} started a novel paradigm by showing  that Sahlqvist {\em correspondence} phenomena can be systematically and generally explained in terms of the same algebraic and order-theoretic principles driving algebraic canonicity, and thus correspondence and canonicity can be unified in a deep way by the same algebraic and order-theoretic root. These insights developed into the so-called {\em unified correspondence theory} \cite{CoGhPa13}, a framework aimed at encompassing the many Sahlqvist-type results into one coherent framework.

The breadth of this work has stimulated many and varied applications. Some are closely related to the core concerns of the theory itself, such as the understanding of the relationship between different methodologies for obtaining canonicity results \cite{PaSoZh14,CP16}, finite embeddability constructions \cite{MoAl}, or of the phenomenon of pseudocorrespondence \cite{CGPSZ14}.  Other, possibly surprising applications include the dual characterizations of classes of finite lattices \cite{FrPaSa14},  the identification of the syntactic shape of axioms which can be translated into analytic rules of  display calculi \cite{GMPTZ}, and the definition of internal Gentzen calculi for the logics of strict implication \cite{MaZhao15}. Finally, the insights of unified correspondence theory have made it possible to determine the extent to which the Sahlqvist theory of classes of normal DLEs can be reduced to the Sahlqvist theory of normal Boolean expansions, by means of G\"{o}del-type translations \cite{CPZ:Trans}, and to  interpret lattice-based modal logic as the epistemic logic of categorization systems (as they are used in social science and management science) \cite{CFPPTW}.


In the light of these developments, we believe it is useful to retell the basics of Sahlqvist correspondence theory in a way which highlights the algebraic and order-theoretic machinery at their core, while at the same time making this core accessible to an audience of logicians who do not need to be familiar with the specific techniques. This is the aim of the present paper.

A key feature of the unified correspondence approach, which is also fundamental to the present exposition,
is the recasting of the original  Sahlqvist correspondence problem  into the setting of the {\em complex algebras} of Kripke frames.\footnote{The complex algebra of a Kripke frame
   is just the powerset algebra of set of possible worlds, endowed with additional unary box and diamond operations defined in terms of the accessibility relation of the given Kripke frame.} In this setting,  the effective computation of the first-order condition corresponding to any given Sahlqvist formula can be analyzed from an {\em algebraic } and {\em order-theoretic} viewpoint.
This analysis makes it possible to draw a neat conceptual distinction between the syntactic shape of formulas, the order-theoretic conditions on their algebraic interpretations, and the mechanism by which these order-theoretic conditions guarantee the success of the Sahlqvist reduction strategy. Thanks to this analysis, the syntactic characterization of Sahlqvist formulas can be explained in terms of the order-theoretic properties of their algebraic interpretation. In its turn, this order-theoretic explanation provides the essential guideline for extending the definition of Sahlqvist formulas to the previously mentioned array of logics.

 The connections between the algebraic and the relational semantics of modal logic and other propositional logics form a mathematically rich and deep theory, the foundations of which were laid by Stone \cite{Sto36}, J\'onsson and Tarski \cite{JT52}, and more recently developed by Goldblatt \cite{Gol89}. However, it is worth stressing again that the treatment of the present paper does not require any previous knowledge or familiarity with this theory, nor with the algebraic approach to logic.\footnote{For the sake of keeping the presentation elementary, in the present paper we mainly focus on Sahlqvist correspondence and do not treat Sahlqvist canonicity in depth. On the other hand, the systematic algebraic treatment of Sahlqvist correspondence is much newer  than  the algebraic treatment of Sahlqvist canonicity, which goes back to J\'onsson \cite{Jo94} and Ghilardi and Meloni \cite{GhMe97}. Indeed, the papers \cite{CoPa10,CFPS} are the first instances known to the authors of the explicitly algebraic treatment of \emph{correspondence} for \emph{classes} of formulas---rather than for isolated instances, as in e.g.\ \cite{DuGePa05}.} 

 In particular,  
  while making it possible to consider and reason about properties with a distinct algebraic and order-theoretic flavour, the environment of complex algebras retains and supports our set-theoretic intuition coming from Kripke frames.

From a pedagogical point of view, the presentation of Sahlqvist correspondence in this environment has proved to be very successful. Indeed, the perspicuity of correspondence arguments couched in terms of complex algebras has already been noted and advocated by Vakarelov, who illustrated it by means of examples in a number of talks and abstracts, of which \cite{Vak2005}  is a good example. In the hands of the second author, this approach has been well received by the participants of various courses, including the 2009 and 2010 installments of the course \emph{Introduction to Modal Logic} at the ILLC (University of Amsterdam), a five-day course at the Istanbul Graduate Summer School in Logic in 2010, and graduate courses and tutorials delivered in 2011 at the Department of Computer Science of the  University of  Milan,  in 2012 at the Institute of Logic and Intelligence of the Southwest University in Chonqing, and in 2016 at the Berkeley-Stanford Circle in Logic and Philosophy.


The present exposition  is closely related to but also very different from  standard textbook treatments (cf.\ e.g.\ \cite{BdRV01, ChZa97,Kra99}), and, without introducing  technicalities such as the Ackermann lemma,  explains in elementary terms the conceptual foundations of unified correspondence theory. We believe that the present treatment  can be useful in making unified correspondence theory accessible to a community of logicians wider than that of the experts in algebraic methods in logic.

\paragraph{Acknowledgements.} We wish to thank the anonymous reviewer for  his/her very careful reading of the paper and for his/her insightful comments which have greatly improved it.


\section{Modal correspondence: a quick introduction}
\label{sec:quick}

The present section begins by collecting the  formal preliminaries to correspondence theory for basic modal logic, up to the standard translation. This is followed by a summary of some  important themes and ideas in correspondence theory, including a bird's-eye view of the most important syntactic classes and proof methods in the literature, and their interconnections. The material in this section by no means constitutes an exhaustive survey of the field, but is only intended to give some necessary background and context to the algebraic perspective on Sahlqvist correspondence developed in Section \ref{section:algebraic correspondence}. It is precisely this perspective on correspondence that has been instrumental in shedding light on  the way the proof methods discussed in the summary relate to one another.

\subsection{Modal logic}

\paragraph{Syntax and semantics.} The \textit{basic modal language}, denoted $\mathrm{ML}$, is defined using a set $\Prop$ of propositional variables, also called atomic propositions, and a unary modal operator $\Diamond$ (`diamond'). The well-formed \textit{formulas} of this language are given by the rule
\[
\varphi::= p \;|\; \bot \;|\; \neg \varphi \;|\; \varphi \vee \varphi \;|\; \Diamond \varphi,
\]
where $p \in \Prop$. The connective $\Box$, the dual of $\Diamond$, is defined as $\Box\varphi := \neg\Diamond\neg\varphi$. The boolean connectives $\wedge$, $\rightarrow$, and $\leftrightarrow$ and also the constant $\top$ are defined as usual. We interpret this language on Kripke frames and models: A \emph{Kripke frame} is a  structure $\f = ( W, R )$ with $W$ a non empty set and $R$ a binary relation on $W$. Augmenting $\f$ with a valuation function $V : \Prop \rightarrow \mathcal{P}(W)$ we obtain a \emph{Kripke model} $\mathcal{M} = ( W, R, V )$.

The \emph{complex algebra}\label{BAO} of a frame $\f = ( W, R )$ is the boolean algebra with operator (BAO)
\[
\f^+ = (\pp(W), \cap, \cup, -_{W}, \varnothing, W, m_R)
\]
where $-_W$ denotes set complementation relative to $W$, and $m_R : \pp(W) \rightarrow \pp(W)$ is given by
\begin{equation}\label{def:mr}
m_R(X) := \{ w \in W | Rwv \mbox{ for some } v \in X \}.
\end{equation}
We also let

\begin{equation}\label{def:lr}
l_{R}(X) := \{ w \in W | Rwv \mbox{ for all }  v \in X\},
\end{equation} or equivalently,
$l_{R}(X) := -_W m_R (-_W X)$.

The perspective we develop in Section \ref{section:algebraic correspondence} is based on $(\p(W), \subseteq)$ being a partial order and the operations of $\f^+$ enjoying certain  properties w.r.t.\ this order.

We now define the semantics of $\mathrm{ML}$ on models and frames via the notion of \emph{meaning function}\footnote{Readers with a background in algebraic logic will immediately recognize  meaning functions as the term functions associated with $\ML$-terms and defined on complex algebras.}. This formulation will be convenient later on in Section \ref{subsec:the general reduction strategy}, when we will develop the discussion on the reduction strategies.

\paragraph{Meaning functions.}
For a formula $\varphi\in \ML$ we write $\varphi = \varphi(p_1,\ldots, p_n)$ to indicate that \emph{at most} the atomic propositions $p_1,\ldots p_n$ occur in $\varphi$. Every such $\varphi$ induces an $n$-ary operation on $\pp(W)$,
\[
\val{\varphi}: \pp(W)^n\to \pp(W),
\]
inductively given by:
\begin{center}
\begin{tabular}{r  l}
$\val{\bot}$ & is the constant function $\varnothing$\\
$\val{p}$ &  is the identity map $\mathit{Id}_{\pp(W)}$\\
$\val{\neg\varphi}$ & is the complementation $W\setminus\val{\varphi}$\\
$\val{\varphi\vee \psi}$  & is the union $\val{\varphi}\cup \val{\psi}$\\
$\val{\Diamond\varphi}$  & is the semantic diamond $m_R(\val{\varphi})$.\\
\end{tabular}
\end{center}
It follows that
\begin{center}
\begin{tabular}{r  l}
$\val{\top}$  & is the constant function $W$\\
$\val{\varphi\wedge \psi}$ & is the intersection $\val{\varphi}\cap\val{\psi}$\\
$\val{\Box\varphi}$  & is the semantic box $l_R(\val{\varphi}).$\\
\end{tabular}
\end{center}
For every formula $\varphi$, the $n$-ary operation $\val{\varphi}$ can be also regarded as a map that takes valuations as
arguments and gives subsets of $\pp(W)$ as its output. Indeed, for every $\varphi \in \Ll$, let
\begin{equation}
\label{notation phiV}
\val{\varphi}(V): = \val{\varphi}(V(p_1),\ldots, V(p_n)).
\end{equation}
Then $\val{\varphi}(V)$ is the {\em extension} of $\varphi$ under the valuation $V$, i.e.\ the set of the states of $(\f, V)$ at
which $\varphi$ is true. Since this happens for all valuations, we can think of $\val{\varphi}$ as the {\em meaning} function of $\varphi$.  We can now define the notion of \emph{truth} of a formula $\phi$ at a point $w$ in a model $\mathcal{M} = ( W, R, V )$, denoted $\mathcal{M}, w \Vdash \phi$, by
\[
\mathcal{M}, w \Vdash \phi \quad \mbox{ iff }\quad  w \in \val{\varphi}(V).
\]
Similarly, \emph{validity} at a point in a frame is given by
\[
\f, w\Vdash \varphi\quad \mbox{ iff }\quad  w \in \val{\varphi}(V) \mbox{ for
every valuation } V \mbox{ on } \f.
\]
The \emph{global versions} of truth and validity are obtained by quantifying universally over $w$ in the above clauses. Thus we have $\mathcal{M} \Vdash \phi$ iff $\val{\varphi}(V) = W$, and $\f\Vdash \varphi$ iff $\val{\varphi}(V) = W$ for
every valuation $V$.

\paragraph{General frames, admissible valuations, and canonicity.}\label{DGF} Appealing as they are, Kripke frames are not adequate to provide uniform completeness results for {\em all} modal logics (the first examples of frame-incomplete modal logic were given by Thomason \cite{Th74}. This issue has been further clarified by Blok \cite{Bl78}).

For uniform completeness, Kripke frames need to be equipped with extra structure. A {\em general frame} is a triple $\mathcal{G} = (W, R, \mathcal{A})$, such that $\mathcal{G}^{\sharp} = (W, R)$ is a Kripke frame, and $\mathcal{A}$ is a sub-BAO of the complex algebra $(\mathcal{G}^{\sharp})^+$ (cf.\ page \pageref{BAO}). For $w\in W$, we define $R[w] = \{v \in W : Rwv\}$ and $R^{-1}[w] = \{v \in W : Rvw\}$. Also, for $S \subseteq W$, we let $R[S] = \bigcup \{R[s] \mid s\in S  \}$ and $R^{-1}[S] = \bigcup \{R^{-1}[s] \mid s\in  S \}$. \label{def:rw} For every $k \in \mathbb{N}$ we define $R^{k}[S]$ by induction on $k$ as follows: $R^{0}[S] = S$, and $R^{k+1}[S] = R[R^{k}[S]]$. \label{notation:rk} An {\em admissible valuation} on $\mathcal{G}$ is a map $v: \mathsf{Prop}\to \mathcal{A}$. Satisfaction and validity of modal formulas w.r.t.\ general frames are defined as in the case of Kripke frames, but by restricting to  admissible valuations.

In fact, the desired uniform completeness can be given in terms of the following proper subclass of  general frames: a general frame $\mathcal{G}$ is {\em descriptive} if $\mathcal{A}$ forms a base for a Stone topology\footnote{A topology $\tau$ on a set $W$ is a {\em Stone} topology if $(W, \tau)$ is compact, and every two distinct points can be separated by some clopen set.} on $W$, in which $R[w]$ is a closed set for each $w\in W$.

In the light of the uniform completeness w.r.t.\ descriptive general frames, to prove that a given modal logic is frame-complete, it is sufficient to show  that its axioms are valid on a given descriptive general frame $\mathcal{G}$ if, and only if, they are valid on its underlying Kripke frame $\mathcal{G}^{\sharp}$. Formulas whose validity is preserved in this way are called {\em canonical}.

\paragraph{The standard translation.} When interpreted on models, modal logic is essentially a fragment of first-order logic, into which we can effectively and straightforwardly  translate it using the so called \emph{standard translation}. In order to introduce it, we need some preliminary definitions.

Let $L_{0}$ be the first-order language with $=$ and a binary relation symbol $R$, over a set of denumerably many individual variables ${\sf VAR} =
\{x_{0},x_{1},\ldots \}$. Also, let $L_{1}$ be the extension of $L_{0}$ with a set of unary predicate symbols $P, Q, P_{0},P_{1}\ldots$,
corresponding to the propositional variables $p, q, p_{0}, p_{1}\ldots$ of $\mathsf{Prop}$. The language $L_2$ is the extension of $L_1$ with universal second-order quantification over the unary predicates $P, Q, P_{0},P_{1} \ldots$

Clearly, Kripke frames are  structures for {\em both} $L_0$ and  $L_2$. Moreover, (modal) models $\mathcal{M} = (W, R, V)$ can be regarded as structures for $L_1$, by interpreting the predicate symbols $P$ associated with any given atomic proposition $p\in \textsf{Prop}$ as the subset $V(p)\subseteq W$.

$\mathrm{ML}$-formulas are translated into $ L_{1}$ by means of the following \textit{standard translation} $ST_x$ from \cite{BdRV01}. Given a first-order variable $x$ and a modal formula $\phi$, this translation yields a first-order formula $\mathrm{ST}_{x}(\phi)$ in which $x$ is the only free variable. $\mathrm{ST}_{x}(\phi)$ is given inductively by
\begin{center}
\begin{tabular}{r c l}
$\mathrm{ST}_x(p)$ & = & $Px$,\\
$\mathrm{ST}_x(\bot)$& = &$x \neq x$, \\
$\mathrm{ST}_x(\neg\varphi)$&=&$\neg(ST_x(\varphi))$, \\
$\mathrm{ST}_x(\varphi \vee \psi )$& = &$ST_x(\varphi) \vee ST_x(\psi)$, \\
$\mathrm{ST}_x(\Diamond \varphi )$& = &$\exists y (Rxy \wedge ST_y(\varphi))$, where $y$ is any fresh variable.
\end{tabular}
\end{center}
The \emph{standard second-order translation} of a modal formula $\phi$ is the $L_2$-formula obtained by universal second-order quantification over all predicates corresponding to proposition letters occurring in $\phi$, that is, the formula $\forall P_1 \ldots \forall P_n \mathrm{ST}_x(\phi)$.

As is well known and easy to check, for every model $(\mathcal{F}, V)$ and state $w$ in it, it holds that $(\mathcal{F}, V), w \Vdash \phi$ iff $(\mathcal{F}, V) \models \mathrm{ST}_{x}(\phi)[x:= w]$. Moreover, $\mathcal{F}, w \Vdash \phi$ iff $\mathcal{F} \models \forall P_1 \ldots \forall P_n \mathrm{ST}_{x}(\phi)[x:= w]$. The analogous global versions of these results are, respectively, $(\mathcal{F}, V)\Vdash \phi$ iff $(\mathcal{F}, V) \models \forall x \mathrm{ST}_{x}$ and $\mathcal{F} \Vdash \phi$ iff $\mathcal{F} \models \forall P_1 \ldots \forall P_n \forall x\mathrm{ST}_{x}(\phi)$.

\subsection{Correspondence}

As seen in the previous subsection, the correspondence between modal languages and predicate logic depends on where one
focuses in the multi-layered hierarchy of relational semantics notions. At the bottom of this
hierarchy lies the  model. At this level, the question of correspondence, at least when
approached from the modal side, is trivial: all modal formulas define first-order conditions
on these structures. This can be made more precise: we have the following elegant theorem by van Benthem:

\begin{theorem}[cf.\ \cite{vanBenthem:CrrspndncThr:84}]
Modal logic is exactly the bisimulation invariant fragment of $L_1$.
\end{theorem}

At the top of the hierarchy, the interpretation of modal languages over relational structures via the notion of validity
turns them into fragments of monadic second-order logic, and rather expressive fragments at
that. Indeed, as Thomason \cite{Tho75} has shown, second-order consequence may be effectively
reduced to the modal consequence over relational structures.

On the other hand, as already indicated, some modal formulas actually define first-order conditions on {\em Kripke frames}. For instance, in the standard second-order translation of any formula $\phi$ which contains no propositional variables (called a \emph{constant formula}), the second-order quantifier prefix is empty. Hence, to mention a concrete example, the standard second-order translation  $ \mathrm{ST}_{x}(\Box \bot)$ is $\forall y (Rxy \rightarrow y \neq y) \equiv \forall y (\neg Rxy)$.

We refer to $\Box \bot$ and $\forall y (\neg Rxy)$ as \emph{local frame correspondents}, since  for all Kripke frames $\f$ and states $w$,
$$\f, w \Vdash \Box \bot\quad \mbox{ iff }\quad \f \models \forall y (\neg Rxy) [x:= w].$$ A modal formula $\phi$ and a first-order sentence $\alpha$ are \emph{global frame correspondents} if $\f \Vdash \phi$ iff $\f \models \alpha$ for all Kripke frames $\f$.

But formulas need not be constant to define first-order conditions: indeed, $p \rightarrow \Diamond p$ and $Rxx$ are local frame correspondents. A short proof of this fact might be instructive: 
Let $\f = (W,R)$ and $w \in W$. Suppose $Rww$, and let $V$ be any valuation such that $(\f, V), w \Vdash p$. Then, since $Rww$, the state $w$ has a successor satisfying $p$, and hence $(\f, V), w \Vdash \Diamond p$. Since $V$ was arbitrary, we conclude that $\f, w \Vdash p \rightarrow \Diamond p$. Conversely, suppose $\neg Rww$, and let $V$ be {\em some} valuation such that $V(p) = \{ w \}$. Then $(\f, V), w \Vdash p$ but $(\f, V),w \not \Vdash \Diamond p$, hence $\f, w \not \Vdash p \rightarrow \Diamond p$.

\noindent The latter direction is an instance of the so called {\em minimal valuation argument},  which  pivots on the fact that some ``first-order definable" minimal element exists in the class of valuations which make the antecedent of  the formula true at $w$. We will take full stock of this observation in Section \ref{subsec:the general reduction strategy}. 

According to the picture emerging from the facts collected so far, 
it is  the correspondence of modal logic and first-order logic on frames  that is open and most interesting. It is here where our efforts are needed in order to try and rescue as much of modal logic as we can from the computational disadvantages of second-order logic. Indeed, there is much that can be salvaged.

In \cite{vanBenthem:CrrspndncThr:84}, van Benthem provides an elegant model-theoretic characterization of the modal formulas which have global first-order correspondents. The constructions involved in this characterization (viz.\ ultrapowers) are infinitary. 
So it would be useful to couple this result with a theorem providing an effective way to check whether a formula is elementary. This would, of course, be
much too good to be true, and indeed, our skepticism is confirmed by Chagrova's theorem:
\begin{theorem}[cf.\ \cite{Chagrov:Chagrova:2006}]\label{Chagrova:Thm}
It is algorithmically undecidable whether a given modal formula is elementary\footnote{A modal formula is \emph{elementary} if it is first-order definable.}.
\end{theorem}
An effective characterization is therefore impossible, but if we are willing to be satisfied with approximations, all is not lost.
Various large and interesting, syntactically defined classes of (locally) elementary formulas are known.

\subsection{Syntactic classes}
A large part of the study of  correspondence between modal and first-order logic has traditionally consisted of the identification of syntactically specified classes of modal formulas which have local frame correspondents.

\begin{description}
\item[Formulas without nesting.] These are the modal formulas in which no nesting of modal operators occur. Their elementarity was proved by van Benthem \cite{vB83}.

\item[Sahlqvist formulas.] This is the archetypal class of elementary modal formulas, due to Sahlqvist \cite{Sah75}. The definition will be given in full in Section \ref{sec:sahlqvist and inductive}. Over the years, many extensions, variations and analogues of this result have appeared, including alternative proofs (e.g.\ \cite{SaVa89}), generalizations to arbitrary modal signatures  (e.g.\ \cite{DeRijke:Venema:95}), variations of the correspondence language (e.g.\ \cite{Ohlbach:Schmidt:97} and \cite{vanBenthem:Fix:Points:2006}), Sahlqvist results for hybrid logics (e.g.\ \cite{TenCate:Marx:Viana}).

    Apart from being elementary, the Sahlqvist formulas have the added virtue of being {\em canonical} (i.e.\ of being valid in the canonical, or Henkin, models of the logics axiomatized by them), and hence, of axiomatizing complete normal modal logics. In other words, logics axiomatized using Sahlqvist formulas are sound and strongly complete with respect to elementary classes of Kripke frames.

\item[Inductive formulas.] The inductive formulas are a generalization of the Sahlqvist class, introduced by Goranko and Vakarelov  \cite{Goranko:Vakarelov:2006}. A certain subclass of inductive formulas will be discussed in Section \ref{sec:sahlqvist and inductive}, together with other syntactic classes.

\item[Modal reduction principles.] A \emph{modal reduction principle} is an $\mathrm{ML}$-formula of the form $Q_1 Q_2 \ldots
    Q_n p \rightarrow Q_{n+1} Q_{n+2} \ldots Q_{n+m} p$ where $0 \leq n, m$ and $Q_{i} \in \{\Box, \Diamond\}$ for $1 \leq i \leq n+m$. Many well known modal axioms take this form, e.g.\ $\Box p \rightarrow p$ (reflexivity), $\Box p \rightarrow \Box \Box p$ (transitivity), $p \rightarrow \Box \Diamond p$ (symmetry), $\Diamond \Box p \rightarrow \Box \Diamond p$ (the Geach axiom), and  $\Box \Diamond p \rightarrow \Diamond \Box p$ (the McKinsey axiom). In \cite{vanBenthem:Reduction:Principles}, van Benthem provides a complete classification of the modal reduction principles that define first-order properties on frames. For example, $\Diamond \Box p \rightarrow \Box \Diamond p$ defines a first-order property and $\Box \Diamond p \rightarrow \Diamond \Box p$ does not.

    In \cite{vanBenthem:Reduction:Principles}, it is also shown that, when interpreted over transitive frames, \emph{all} modal reduction principles define first-order properties.

\item[Complex formulas.] This class was introduced by  Vakarelov  \cite{Vakarelov:AiML:2003}. Complex formulas can be seen as substitution instances of  Sahlqvist formulas obtained through the substitution of certain elementary disjunctions for propositional variables. The resulting formulas may violate the Sahlqvist definition. 

\end{description}

\subsection{Algorithmic strategies}

The standard proof of the elementarity of the Sahlqvist formulas takes the form of an algorithm, known as the Sahlqvist-van Benthem algorithm, which computes first-order correspondents for the members of this class. However, the syntactic definition of the Sahlqvist formulas is taken as primary. Approaches which take the algorithm as primary  have gained momentum in recent years.

One approach is to apply algorithms for second-order quantifier elimination to the second-order translations of modal formulas. Two notable examples are the algorithms  SCAN \cite{GabbayOlbach} and DLS \cite{Szalas1}. SCAN is a resolution based algorithm and can successfully compute the first-order frame correspondent of every Sahlqvist formula \cite{Goranko:SCAN:Complete}. The same is true of DLS \cite{Conradie:DLS}, which, in contrast, is based on Ackermann's lemma.

The algorithm SQEMA \cite{Conradie:et:al:SQEMAI} was designed specifically for modal logic. It computes first-order frame correspondents for all inductive (and hence Sahlqvist) formulas, among others. SQEMA is based on a series of syntactic manipulations of modal formulas which are aimed at achieving a system of implications in which propositional variables can be `solved for' and then eliminated via suitable substitutions justified by Ackermann's lemma.  SQEMA guarantees the canonicity of all formulas on which it succeeds (canonicity will be discussed in more detail in the next subsection), and has been extended and developed in a series of papers by the same authors \cite{SQEMA2,SQEMA3,SQEMA4,CoGoVa10}. The algorithm ALBA \cite{CoPa10}, which first appeared as the analogue of SQEMA in the setting of distributive modal logic (DML) \cite{GNV}, has been generalized so as to account for logical settings in various signatures (e.g.\ non-distributive logics \cite{CoPa11}, non-normal (modal) logics \cite{PaSoZh14a}, mu-calculus on  bi-intuitionistic propositional base \cite{CFPS}, hybrid logics \cite{ConRob}, many-valued logics \cite{manyval}), and has also been applied to environments in which correspondence might not be defined (e.g.\ constructive canonicity \cite{CP16,CCPZ}), to different issues than correspondence proper (e.g.\ pseudocorrespondence \cite{CGPSZ14}), even shedding light on core issues in structural proof theory (the  computation of analytic rules from axioms \cite{GMPTZ, MaZhao15}). These latter applications have significantly expanded the conceptual significance of ALBA beyond its original purpose as a `calculus for correspondence'. In Section \ref{sec:glimpse}, we will give an informal outline  of how ALBA works in connection with the insights presented in Section \ref{section:algebraic correspondence}.

\subsection{Algebraic approaches to Sahlqvist canonicity}

As mentioned earlier, Sahlqvist formulas have the added benefit of being canonical (cf.\ page \pageref{DGF}), and hence complete w.r.t.\  the elementary class of frames defined by their correspondents. In the literature, there exist three prominent algebraic approaches to Sahlqvist theory, which are motivated by the study of modal canonicity. They  can be respectively traced back to Sambin and Vaccaro \cite{SaVa89},  J\'onsson \cite{Jo94} and Ghilardi and Meloni \cite{GhMe97}.  
Recall that proving that a given modal formula is canonical involves proving that its validity transfers from any descriptive general frame to its underlying Kripke frame.

Sambin and Vaccaro \cite{SaVa89}  gave a simplified proof of the Sahlqvist canonicity theorem using order-topological methods.  Their strategy, which is sometimes referred to as {\em canonicity-via-correspondence}, makes use of the existence of a first-order correspondent for any given Sahlqvist formula {\em in the first-order language of the underlying frame}.

\begin{center}
\begin{tabular}{l c l}\label{table:U:shape}
$\mathcal{G}\Vdash\phi$ & &$\mathcal{G}^{\sharp}\Vdash\phi$\\
 &&\\
$\ \ \ \Updownarrow$ & & $\ \ \ \Updownarrow$\\
&&\\
$\mathcal{G}\vDash$FO$(\phi)$

&\ \ \ $\Leftrightarrow$ \ \ \ &$\mathcal{G}^{\sharp}\vDash$FO$(\phi)$.\\

\end{tabular}
\end{center}
The crucial observation is that the truth of a first-order sentence in a Kripke frame, seen as a relational model, is independent of  the assignments being admissible or arbitrary, as illustrated by the horizontal bi-implication in the diagram above. From this, and the correspondence (represented by the two vertical bi-implications in the diagram) the canonicity result follows.

Goranko and Vakarelov \cite{Goranko:Vakarelov:2006}  give a proof of canonicity for inductive formulas using a similar strategy to that of Sambin and Vaccaro.

The second algebraic approach to Sahlqvist theory heavily relies on the theory of canonical extensions \cite{JT52}, and focuses on canonicity independently of correspondence. In his seminal work, J\'{o}nsson \cite{Jo94} proved the canonicity of Sahlqvist identities using the fact that the operations interpreting the logical connectives can be extended from a given BAO to its canonical extension, and then using the order-theoretic properties of these extensions and of their resulting compositions. Independently,  Ghilardi and Meloni \cite{GhMe97} gave a constructive proof of canonicity in a setting of  bi-intuitionistic (tense) modal logic using filter and ideal completions. Suzuki \cite{Suz11} extends Ghilardi and Meloni's method and proves canonicity for Sahlqvist formulas in a setting of substructural modal logics. 

Following \cite{Jo94}, Gehrke, Nagahashi and Venema \cite{GNV} proved the canonicity of Sahlqvist inequalities in the language of distributive modal logic (DML).\footnote{See also the emendations discussed in \cite{PaSoZh14} to the proof of \cite[Theorem 5.6]{GNV}.} In \cite{CoPa10},   duality theory is used to unify the algorithmic approach of \cite{Conradie:et:al:SQEMAI} and the algebraic techniques of \cite{Jo94,GNV}; in particular, the DML-counterparts of  inductive formulas are defined, and their canonicity is proved `via correspondence'. Perhaps more importantly, the methodology of \cite{CoPa10} shows that  the mechanism of correspondence can be explained in terms of  those same algebraic and order-theoretic properties which have been recognized as the main driver of canonicity results, thus conceptually unifying the two sides of Sahlqvist theory.

Recent papers \cite{PaSoZh14,PaSoZh14a} extend the techniques in \cite{Jo94,GNV} to the inductive inequalities of distributive modal logic and to the Sahlqvist inequalities of a very general setting  of regular (i.e.\ possibly non-normal)  modal logics on a weaker than classical propositional base. The techniques in \cite{Jo94,GNV} do not straightforwardly generalize to inductive inequalities. This hurdle was overcome  in \cite{PaSoZh14}, where ALBA is used for canonicity in a novel way which, interestingly, does not involve correspondence. 
Even more recent work  \cite{CP16,CCPZ} generalizes the ALBA methodology for canonicity to a constructive setting algebraically captured by general lattice expansions, unifying the  canonicity techniques pioneered by Ghilardi-Meloni and Sambin-Vaccaro, and also covering  fixed-point logics.

\subsection{Sahlqvist via translation}
The last approach we mention aims at obtaining Sahlqvist-type correspondence and canonicity for logics with weaker propositional bases than the classical---e.g.\ intuitionistic (modal) logic, distributive modal logic---as a byproduct of Sahlqvist correspondence for classical modal logic, via translations such as the G\"odel-Tarski.  In \cite{WZ97,WZ98}, this idea has been pursued to achieve results  closely related to correspondence, and in \cite{GNV}, this was the route through which the correspondence result for Sahlqvist DML-inequalities was proven. More recently, \cite{vBBHol} develops a full and faithful translation of the basic normal modal logic into a bimodal logic which is semantically motivated by the interplay between Kripke semantics and possibility semantics. This translation makes it possible to transfer results of classical Sahlqvist correspondence to the possibility semantics setting. In \cite{CPZ:Trans},  the feasibility of the route `via translation'  for obtaining Sahlqvist correspondence and canonicity results is systematically investigated in various  settings of normal expansions of algebras, from normal bi-Heyting algebra expansions to normal distributive lattice expansions  (DLEs). The starting point of this analysis is an order-theoretic reformulation
of the main semantic property of the G\"odel-Tarski translation as a certain adjunction situation. The main conclusions of this analysis are
twofold. On the positive side, it is possible to prove, via translation, correspondence results for inductive inequalities in arbitrary signatures of normal of normal distributive lattice expansions.
On the negative side, canonicity results can be proven via translation only only in the special setting of normal bi-Heyting algebra expansions. Therefore, the existing unified correspondence techniques remain the most economical route to these results.

\section{Sahlqvist formulas and atomic inductive  formulas}
\label{sec:sahlqvist and inductive}
The aim of the present section is to define the syntactic classes that we are going to treat in Section \ref{section:algebraic correspondence}.
For the reader's convenience, 
we will do this in a hierarchical form  which is conceptually akin to the treatment in \cite[Section 3.6]{BdRV01}, although it is slightly different from it, to better fit our account in the following section. 
 In particular, we also introduce the class of {\em atomic inductive formulas}, which properly extends the Sahlqvist class. 
To further set the stage for the treatment in Section \ref{section:algebraic correspondence}, we will also show that the formulas in each of these classes can be equivalently rewritten in certain normal forms. We confine our presentation to the basic modal language.

\subsection{Sahlqvist implications and atomic inductive implications}
\label{ssec:sahlqvist and linear ind}

\paragraph{Closed and uniform formulas.} The \textit{closed} modal formulas are those that contain no proposition letter.
An occurrence of a proposition letter $p$ in a formula $\phi$ is a {\em positive} ({\em negative}) if it is under the scope of an even (odd) number of negation signs. (To apply this definition correctly one of course has to bear in mind the negation signs introduced by the defined connectives $\rightarrow$ and $\leftrightarrow$.) A formula $\varphi$ is \textit{positive} in $p$ (\textit{negative} in $p$) if all occurrences of $p$ in $\varphi $ are positive (negative).

A proposition letter occurs \emph{uniformly} in a formula if it occurs only positively or only negatively. A modal formula is \emph{uniform} if all the propositional letters it contains occur uniformly. Let $\mathsf{UF}$ be the class of uniform formulas.

\paragraph {Very simple Sahlqvist implications.} 
A \emph{very simple Sahlqvist antecedent} is a formula built up from $\top$, $\bot$, negative formulas and proposition letters, using $\vee$, $\wedge$ and $\Diamond$. A \textit{ 
very simple Sahlqvist implication} is an implication $\varphi \rightarrow \psi$ in which $\psi$ is positive and $\phi$ is a 
very simple Sahlqvist antecedent.
Let $\mathsf{VSSI}$ be the class of very simple Sahlqvist implications.

\paragraph{Sahlqvist implications.} A \emph{boxed atom} is a propositional variable preceded by a (possibly empty) string of boxes, i.e.\ a formula of the form $\Box^{n} p$ where $n \in \mathbb{N}$ and $p \in \mathsf{Prop}$. 
A \emph{Sahlqvist antecedent} is a formula built up from $\top$, $\bot$, negative formulas and boxed atoms, using $\vee$, $\wedge$ and $\Diamond$. A \emph{
Sahlqvist implication} is an implication $\phi \rightarrow \psi$ in which $\psi$ is positive and $\phi$ is a 
Sahlqvist antecedent. Let $\mathsf{SI}$ be the class of  Sahlqvist implications.

\paragraph{Atomic inductive implications.} The following definition is a special case of \cite[Definition 27]{Goranko:Vakarelov:2006}. Let $\sharp$ be a symbol not belonging to $\mathrm{ML}$. An \emph{atomic box-form of $\sharp$} in $\mathrm{ML}$ is defined
recursively as follows:
\begin{enumerate}
\item for every $k\in \mathbb{N}$, $\Box^k\sharp$ is an atomic box-form of $\sharp$;
\item If $B(\sharp)$ is an atomic box-form of $\sharp$, then for any proposition letter $p$, $\Box(p \rightarrow B(\sharp))$ is an atomic box-form of $\sharp$.
\end{enumerate}
Thus, atomic box-forms of $\sharp$ are of the type
\[
\Box(p_{0}\rightarrow \Box(p_{1}\rightarrow \ldots \Box(p_{n}\rightarrow \Box^k\sharp)\ldots)),
\]
where the $p$s are are not necessarily different.

By substituting any propositional variable $p\in \mathsf{Prop}$ for $\sharp$ in an atomic
box-form $B(\sharp)$ we obtain an \emph{atomic box-formula of $p$}, namely
$B(p)$. The occurrence of the variable $p$ substituted for $\sharp$ is called the \emph{head}
of $B(p)$ and every other occurrence of a variable in $B(p)$ (including every occurrence of $p$ which is not the head) is
called \emph{inessential}.
An \emph{atomic regular antecedent} is a formula built up from $\top$, $\bot$, negative formulas, and atomic box-formulas, using $\vee$, $\wedge$ and $\Diamond$.

\noindent The \emph{dependency digraph} of a set of box-formulas
$\mathcal{B}=\{B_{1}(p_{1}),\ldots , B_{n}(p_{n})\}$ is the directed
graph $G_{\mathcal{B}}=\left\langle V,E\right\rangle $ where $V
=\{p_{1},\ldots ,p_{n}\}$ is the set of heads of members of $\mathcal{B}$,
and $E$ is a binary relation on $V$ such that $p_{i} E p_{j}$ iff
$p_{i}$ occurs as an inessential variable in a box formula in $\mathcal{B}$ with
head $p_{j}$. A digraph is \emph{acyclic} if it contains no directed cycles or loops.
Note that the transitive closure of the edge relation $E$ of an acyclic digraph is a strict partial order, i.e.\ it is irreflexive and transitive, and consequently also antisymmetric. The dependency digraph of a formula $\phi$ is the dependency digraph of the set of
box-formulas that occur as subformulas of $\phi$.

An \emph{
atomic inductive antecedent} is an 
atomic regular antecedent with an acyclic dependency digraph.
An \emph{
atomic regular (resp.\ inductive) implication} is an implication $\phi \rightarrow \psi$ in which $\psi$ is positive and $\phi$ is an 
atomic regular (resp.\ inductive) antecedent. Let $\mathsf{AII}$ be the class of atomic inductive implications.

\begin{example}
Consider the following formulas:
\begin{eqnarray*}
\phi_1 &:= &p\wedge \square (p \rightarrow \square q) \rightarrow \Diamond \square \square q,\\
\phi_2 &:= & \Diamond \Box p \wedge \Diamond ( \Box(p \rightarrow q) \vee  \Box (p \rightarrow \Box \Box r)) \rightarrow  \Diamond \Box (q \vee \Diamond r)\\
\phi_3 &:= & \Diamond(\Box( p \rightarrow \Box \Box q) \lor \Box(q \rightarrow \Box p)) \rightarrow \Diamond \Box p.
\end{eqnarray*}

$\phi_1$ is an atomic inductive implication which is not a Sahlqvist implication. The antecedent is the conjunction of the atomic box-formulas $p$ and $\square (p \rightarrow \square q)$. The dependency digraph over the set of heads $\{p,q\}$ has only one edge, from $p$ to $q$, and thus linearly orders the variables.

$\phi_2$ is an atomic inductive implication. Its dependency digraph has three vertices $p$, $q$, and $r$, and arcs from $p$ to $q$ and from $p$ to $r$.

$\phi_3$ is an atomic regular but not inductive implication. Its dependency digraph contains a 2-cycle on vertices $p$ and $q$.
\end{example}
\begin{remark}
As mentioned earlier, the development of the present paper extends beyond the usual textbook treatments of Sahlqvist theory by accounting for atomic inductive formulas. However, it does not cover exhaustively  the state-of-the-art in Sahlqvist theory, which is the class of so-called {\em inductive formulas},  defined analogously to atomic inductive formulas by allowing {\em arbitrary nesting of boxes} and {\em arbitrary positive terms} instead of individual variables in the antecedent of implications  in box-formulas, and then requiring that the corresponding dependency diagraph be acyclic   (cf.\ \cite[Definition 27]{Goranko:Vakarelov:2006}). The definition of atomic inductive formulas is conveniently simpler than the general one, but as will be discussed in Section \ref{section:algebraic correspondence}, it displays the order-theoretic behaviour which separates inductive formulas from Sahlqvist formulas.
\end{remark}
\subsection{Sahlqvist formulas and atomic inductive formulas}

\paragraph{Sahlqvist formulas.}
A \textit{Sahlqvist formula} is a formula that is built up from Sahlqvist implications by freely applying boxes, conjunctions and disjunctions\footnote{Actually, in the definition in \cite{BdRV01} the application of disjunction is restricted, and is only allowed  between formulas that do not share any propositional variables. Proposition \ref{Sahlqvist:Antecedent:Prop} shows that this restriction is unnecessary in the Boolean case. More about this in Remark \ref{rem: distributive vs boolean SF}.}. Let $\mathsf{SF}$ be the class of Sahlqvist formulas.

\paragraph{Atomic inductive formulas.} An \textit{atomic inductive formula} is a formula that is built up from atomic inductive implications by freely applying boxes, conjunctions, and disjunctions. Let $\mathsf{AIF}$ be the class of atomic inductive formulas.

The correspondence results for Sahlqvist and atomic inductive formulas can be respectively reduced to the correspondence results for Sahlqvist and atomic inductive {\em implications}. This is an immediate consequence of the following proposition:
\begin{prop}\label{Sahlqvist:Antecedent:Prop}
Let $\Phi\in\{\mathsf{SF}, \mathsf{AIF}\}$. Every $\phi\in \Phi$ is semantically equivalent to a negated
$\Phi$-antecedent, and hence to a  $\Phi$-implication\footnote{By a $\Phi$-{\em antecedent} we mean a Sahlqvist antecedent if $\Phi = \mathsf{SF}$ or an atomic inductive antecedent if $\Phi = \mathsf{AIF}$; by a $\Phi$-{\em implication} we mean a Sahlqvist implication if $\Phi = \mathsf{SF}$ or an atomic inductive implication if $\Phi = \mathsf{AIF}$.}.
\end{prop}
\begin{proof}
Fix a formula  $\phi\in \Phi$, and let $\phi'$ be the formula
obtained from $\neg \phi$ by importing the negation over all
connectives. Since $\phi \equiv \phi' \rightarrow \bot$, it is enough to show that $\phi'$ is a $\Phi$-antecedent, in order to prove the statement. This is done by induction on the construction of $\phi$ from
$\Phi$-implications. If $\phi$ is a $\Phi$-implication $\alpha \rightarrow
{\sf Pos}$, negating and rewriting it as $\alpha \land \neg {\sf
Pos}$ already turns it into a $\Phi$-antecedent. If $\phi =
\Box \psi$, where $\psi$ satisfies the claim, then $\neg \phi
\equiv \Diamond \neg \psi$ hence the claim follows for $\phi$,
because $\Phi$-antecedents are closed under diamonds. Likewise,
if $\phi = \psi_1 \land \psi_2$, where $\psi_1$ and $\psi_2$
satisfy the claim, then $\neg \phi \equiv \neg \psi_1 \lor \neg
\psi_2$ hence the claim follows for $\phi$, because
$\Phi$-antecedents are closed under disjunctions.
The case of $\phi = \psi_1 \lor \psi_2$ is completely analogous.
\end{proof}
%
%

\subsection{Definite implications}
In the previous subsection, we saw how the correspondence results for formulas in $\mathsf{SF}$ and in $\mathsf{AIF}$ can be respectively reduced to the correspondence results for formulas in $\mathsf{SI}$ and in $\mathsf{AII}$. In their turn, the latter correspondence results can be respectively reduced to the correspondence results for the  subclasses of their {\em definite implications}. These are defined
by forbidding the use of disjunction, except within negative formulas, in the building of  antecedents. 
To be precise:
\begin{definition}

\begin{itemize}
\item A \emph{definite very simple Sahlqvist antecedent} is a formula built up from $\top$, $\bot$, negative formulas and propositional letters, using only $\wedge$ and $\Diamond$.
\item A \emph{definite Sahlqvist antecedent} is a formula built up from $\top$, $\bot$, negative formulas and boxed atoms, using only $\wedge$ and $\Diamond$.
\item A \emph{definite atomic regular antecedent} is a formula built up from $\top$, $\bot$, negative formulas, and atomic box formulas, using only $\wedge$ and $\Diamond$. A \emph{definite atomic inductive antecedent} is a definite atomic regular antecedent with an acyclic dependency digraph.
\end{itemize}
Let $\Phi\in \{\mathsf{VSSI}, \mathsf{SI}, \mathsf{AII}\}$. Then $\phi\to \psi\in \Phi$ is a  \emph{definite $\Phi$-implication} if $\phi$ is a definite $\Phi$-antecedent.
\end{definition}
In the next section, we will be able to confine our attention w.l.o.g.\ to the definite implications in each class, thanks to Fact \ref{fact:conjunction and local correspondence} and to Proposition \ref{Ind:Eqiv:Definite:Prop} below.
\begin{fact}\label{fact:conjunction and local correspondence}
If $\phi_i\in \ML$ locally corresponds to $\alpha_i(x)\in L_0$ for $1\leq i\leq n$, then $\bigwedge_{i=1}^n \phi_i$ locally corresponds to $\bigwedge_{i=1}^n \alpha_i(x)$.
\end{fact}
\begin{prop}\label{Ind:Eqiv:Definite:Prop}
Let  $\Phi\in \{\mathsf{VSSI}, \mathsf{SI}, \mathsf{AII}\}$. Every $\phi\in \Phi$ is equivalent to a conjunction of definite implications in $\Phi$.
\end{prop}
\begin{proof}
Note that $\phi$ can be equivalently rewritten as a conjunction of definite  implications in $\Phi$ by exhaustively distributing $\Diamond$ and $\wedge$ over $\vee$ in the antecedent, and then applying the equivalence $A \vee B \rightarrow C \equiv (A \rightarrow C) \wedge (B \rightarrow C)$.
\end{proof}

\begin{remark}
\label{rem: distributive vs boolean SF}
As we mentioned already in Section \ref{sec:quick}, correspondence theory has been extended to logics with a weaker than classical propositional base. Accordingly, as will be clear from the current account, the correspondence mechanism is independent of the Boolean setting we are in.  However, there is one single point in our presentation in which we took advantage of the specific properties of the classical setting, namely Proposition \ref{Sahlqvist:Antecedent:Prop}. Thanks to classical negation, we are able to pack (i.e.\ prove the semantical equivalence of) any Sahlqvist/atomic inductive formula in (to) {\em one} Sahlqvist/atomic inductive implication. In settings where  classical negation is not available, this cannot be done. Still, Sahlqvist formulas can be defined as in \cite{BdRV01}, i.e.\  allowing the application of disjunction only  between formulas that do not share any proposition letters, and the correspondence result for this class can still be reduced to the correspondence result for Sahlqvist implications, thanks to Fact \ref{fact:conjunction and local correspondence} and the following facts (cf.\ \cite[Lemma 3.53]{BdRV01}):
\begin{enumerate}
\item If $\phi\in \ML$ locally corresponds to $\alpha(x)\in L_0$, then for every $k\in \mathbb{N}$, $\Box^k\phi$ locally corresponds to $\forall y(x R^ky\to\alpha(y))$.
\item If $\phi_i\in \ML$ locally corresponds to $\alpha_i(x)\in L_0$ for $1\leq i\leq n$, and for every $1\leq i, j\leq n$ if $i\neq j$ then $\phi_i$ and $\phi_j$ do not have proposition letters in common, then $\bigvee_{i = 1}^n \phi_i$ locally corresponds to $\bigvee_{i = 1}^n \alpha(x)_i$.
\end{enumerate}
\end{remark}

\section{Algebraic correspondence}
\label{section:algebraic correspondence}

The present section is the heart of the paper. In it, we will proceed incrementally and give the algebraic correspondence argument for the definite formulas of each class defined in Section \ref{ssec:sahlqvist and linear ind}. More details on the methodology are given in the next subsection.
\subsection{The general reduction strategy}
\label{subsec:the general reduction strategy}
Our starting point is the well known fact, already mentioned above, that any modal formula $\varphi$ {\em locally} corresponds to its standard second-order translation, i.e.
\begin{equation}\label{2ndOrderCorrespond:Eqn}
\f, w\Vdash \varphi \quad \mbox{ iff }\quad \f\models \forall P_1\ldots\forall P_n
ST_x(\varphi)[x :=w].
\end{equation}
We are interested in strategies that produce a semantically equivalent first-order condition out of the default local second-order correspondent of $\varphi$ on the right-hand side of (\ref{2ndOrderCorrespond:Eqn}).

A large and natural class of formulas for which, by definition, this is possible is introduced  by van Benthem \cite{vB83}: 
\begin{definition}\label{vanBentFmls:Def}
The class of \emph{van Benthem-formulas}\footnote{this name first appears in \cite{Conradie:Gor:Vak:AiML:04}.} consists of those formulas $\phi \in \mathrm{ML}$ for which $\forall P_1 \ldots \forall P_n ST_x(\varphi)$ is equivalent to $\forall P'_1 \ldots \forall P'_n ST_x(\varphi)$ where the quantifiers $\forall P'_1 \ldots \forall P'_n$ range, not over all subsets of the domain, but only those that are definable by means of $L_0$-formulas.
\end{definition}
%
 The van Benthem-formulas are the designated targets of the reduction strategy in its most general form. To see this,  for every frame $\f = (W, R)$, let $$\mathsf{Val_{L_0}}(\f) = \{V':\mathsf{Prop}\to \p(W)\mid V'(p)\mbox{ is } L_0\mbox{-definable for every } p\in \mathsf{Prop} \}.$$ This is the set of the {\em tame} valuations on $\f$. Using the notation introduced in (\ref{notation phiV}), if $\varphi\in \ML$ is a van Benthem formula, then the following chain of equivalences holds for every $\f$ and  $w\in W$:

\begin{equation}
\label{crucial iff}
\begin{tabular}{r c l}
$\f, w\Vdash \varphi$ 
& iff & $w\in \val{\varphi}(V)$ for every $V$ on $\f$\\
& iff & $w\in \val{\varphi}(V')$ for every $V'\in \mathsf{Val_{L_0}}(\f).$\\
\end{tabular}
\end{equation}

\begin{theorem}\label{vanBentFml:Theorem}
Every van Benthem-formula has a local first-order frame correspondent.
\end{theorem}
\begin{proof}
Let $\phi$ be a van Benthem-formula and let $\Sigma$ be the set of all $L_0$ substitution instances of $\mathrm{ST}_x(\phi)$, i.e.\ the set of all formulas obtained by substituting $L_0$-formulas $\alpha(y)$ for occurrences $P(y)$ of predicate symbols in $\mathrm{ST}_x(\phi)$. Clearly, $\forall \overline{P} \mathrm{ST}_x(\phi) \models \Sigma [x:= w]$, where $\overline{P}$ is the vector of all predicate symbols occurring in $\mathrm{ST}_x(\phi)$. Also, since $\phi$ is a van Benthem-formula, $\Sigma \models \forall \overline{P} \mathrm{ST}_x(\phi) [x := w]$.
Then $\Sigma \models \mathrm{ST}_x(\phi) [x := w]$, and since this is a first-order consequence, we may appeal to the compactness theorem to find some finite subset $\Sigma' \subseteq \Sigma$ such that $\Sigma' \models \mathrm{ST}_x(\phi) [x := w]$.

We claim that $\Sigma' \models \forall \overline{P} \mathrm{ST}_x(\phi) [x := w]$. Indeed, let $\mathcal{M}$ be any $L_1$-model such that $\mathcal{M} \models \Sigma'[x := w]$. Since the predicate symbols in $\overline{P}$ do not occur in $\Sigma'$, every $\overline{P}$-variant of $\mathcal{M}$\footnote{By which we mean a model which difers from $\mathcal{M}$ only possibly in the interpretation of the predicates in $\overline{P}$.} also models $\Sigma'$, and hence also $\mathrm{ST}_x(\phi)$. It follows that $\mathcal{M} \models \forall \overline{P} \mathrm{ST}_x(\phi)[x:= w]$. Thus we may take $\bigwedge \Sigma'$ as a local first-order frame correspondent for $\phi$.
\end{proof}

However, 
relying on compactness, as it does, Theorem \ref{vanBentFml:Theorem} is of little use if we want to explicitly calculate the first-order correspondent for a given $\phi\in \ML$, or devise an algorithm which produces first-order frame correspondents for each member of a given {\em class} of modal formulas; therefore a more refined strategy is in order, the development of which is the core of correspondence theory.

Each class of modal formulas of Subsection \ref{ssec:sahlqvist and linear ind} 
is defined so as to guarantee that the second `iff' (i.e.\ the non trivial one) of (\ref{crucial iff}) can be proved not just for  $V'$ ranging arbitrarily  over $\mathsf{Val_{L_0}}(\f)$ but rather ranging over a much more restricted and nicely defined subset  of it. Moreover,  each of these subsets of tame valuations is defined in such a way as to enable the algorithmic generation of the first-order correspondents of the members of the class of formulas it is paired with. More specifically, as we will see next,  the following pairings hold between classes of formulas and subsets of tame valuations (recall notation introduced in ):
\begin{center}
\begin{tabular}{r| l}
$\mathsf{UF}$ & $V': \mathsf{Prop}\to \{\varnothing, W\}$\\
$\mathsf{VSSI}$ & $V': \mathsf{Prop}\to \p_{fin}(W)$\\
$\mathsf{SI}$ & $V': \mathsf{Prop}\to \bigcup_{j = 1}^n\{R^{k_j}[x_j]\mid k_j \in \mathbb{N}, x_j \in W \}$,
\end{tabular}
\end{center}
where $\p_{fin}(W)$ is the collection of finite subsets of $W$, and the notation $R^{k_j}[x_j]$ has been introduced on page \pageref{notation:rk}. So far, our account has been faithful to the textbook exposition, albeit with slightly different notation. The  algebraic treatment which we are about to introduce crucially provides an intermediate step which clarifies the textbook account:  each class of modal formulas of Section \ref{ssec:sahlqvist and linear ind} 
is defined so as to guarantee that, for every formula $\phi$ in the given class, the meaning function $\val{\phi}$ enjoys certain  purely order-theoretic properties which ensures that the second crucial `iff' can be proved for $V'$ ranging in the corresponding subclass of tame valuations (the definition of which, as we already mentioned, underlies the algorithmic generation of the first-order correspondent of $\phi$).
We start to see how this works  in the next subsection.



\subsection{Uniform and closed formulas}

\paragraph{The reduction strategy.}  Among all the first-order definable  valuations  $V$ on $\f$, the simplest ones are those which assign $W$ or $\emptyset$ to each propositional variable. Indeed, let $V_0$ be such a valuation and suppose that the following were equivalent for the modal formula $\varphi$:
\begin{center}
\begin{tabular}{r c l}
$\f, w\Vdash \varphi$ 
& iff & $w\in \val{\varphi}(V)$ for all $V$ on $\f$\\
& iff & $w\in \val{\varphi}(V_0)$\\
\end{tabular}
\end{center}
This would in turn mean that
\begin{center}
\begin{tabular}{c l}
& $\f\models \forall P_1\ldots\forall P_n ST_x(\varphi){[x := w]}$\\
 iff &
$\f\models ST_x(\varphi){[x := w, P_1 := V_0(p_1), \ldots, P_n := V_0(p_n)]}$\\
\end{tabular}
\end{center}
Therefore, we could equivalently transform the formula above into a first-order formula by replacing each occurrence $P_i z$ with either $z \neq z$ if $V_0(p_i)=\emptyset$, or with $z=z$ if $V_0(p_i)= W$. This is enough to effectively generate the first-order correspondent of $\varphi$.

\paragraph{Order-theoretic conditions.}

For which formulas $\phi$ would it be possible to implement the reduction strategy outlined above? 
The answer to this question can be given in purely order-theoretic terms: 

\begin{prop}\label{sahlqvist for uniform}
Let $( X_i, \leq)$, $i = 1,\ldots, n$, and $(Y, \leq)$ be posets. Let
each $X_i$ have a maximum, $\top_i$, and a minimum, $\bot_i$. Let $f:
X_1\times \cdots \times X_n \to Y$. If $f$ is either order-preserving or order-reversing in each
coordinate, then the minimum of $f$ exists and is $f(c_1, \ldots
c_n)$, where, for every $i$, $c_i = \bot_i$ if $f$ is order
preserving in the $i$-th coordinate, and $c_i = \top_i$ if $f$ is
order-reversing in the $i$-th coordinate.
\end{prop}
%
%
\begin{cor}
\label{ord th for uniform}
\em{For every $\varphi\in \ML$, if $\val{\varphi}: \p(W)^n\to \p(W)$ is order-preserving or order-reversing in each coordinate, then the following are equivalent:
\begin{enumerate}
\item  $\forall V[w\in \val{\varphi}(V)].$
\item  $w \in \val{\varphi}(V_0)$, where $V_0(p_i) = W$ if $\val{\phi}$ is order-reversing in its $i$th coordinate, and $V_0(p_i) = \emptyset$ if $\val{\phi}$ is order-preserving in in its $i$th coordinate.
\end{enumerate}}
\end{cor}
\begin{proof}
$(1 \Rightarrow 2)$ Clear. $(2\Rightarrow 1)$ It follows from Proposition \ref{sahlqvist for uniform} that $\val{\phi}(V_0)$ is the minimum of $\val{\phi}$ and hence $\val{\phi}(V_0) \subseteq \val{\phi}(V)$ for every valuation $V$.
\end{proof}
\paragraph{Syntactic conditions.} Now that we have the reduction strategy and sufficient order-theoretic conditions for the strategy to apply, it only remains to verify that these conditions are met by the uniform formulas. And indeed, the following proposition  can be easily shown by induction on $\phi$:

\begin{prop}
\label{orderpreserving operations}
If $\varphi\in \ML$ is a uniform formula, then $\val{\varphi}$ is order-preserving (reversing) in those coordinates corresponding to propositional variables in which $\phi$ is positive (negative).
\end{prop}
%
%
%
\begin{example}
Let us consider the uniform formula $\Box \Diamond  p$. The minimal valuation for $p$ is $V_0(p)=\emptyset$, since the formula is positive and hence order-preserving in $p$. The standard translation of this formula gives
\begin{center}
\begin{tabular}{c l}
&$\f\models \forall P \; \forall y (Rxy \rightarrow \exists z(Ryz \wedge Pz)[x := w]$\\
iff &$\f\models  \; \forall y (Rxy \rightarrow \exists z(Ryz \wedge P^{0}z)[x:= w]$
\end{tabular}
\end{center}
where the predicate $P^{0}z$ can be replaced with $z\neq z$ giving a first-order equivalent formula  $\forall y (Rxy \rightarrow \exists z(Ryz \wedge z\neq z)$ which simplifies  to $\forall y (\neg Rxy)$.
\end{example}

To sum up: although the uniform formulas and their accompanying 
valuations are extremely simple, the key features of our account are already present:

\begin{itemize}
\item
 first, the subclass of  tame valuations is identified, using which the desired first-order correspondent can be effectively computed;
 \item
 second, the order-theoretic properties are highlighted, which guarantee the crucial preservation of equivalence;
 \item
 third, the syntactic specification of the formulas $\phi$ of the given class guarantees that their associated meaning functions $\val{\phi}$ meet the required order-theoretic properties.
\end{itemize}


\paragraph{Non-uniform formulas and `minimal valuation' argument.}
\label{sssec:non-uniform}
The discussion above also shows that every uniform formula is locally equivalent on frames to some closed formula (which is obtained by replacing every positive variable with $\bot$ and every negative variable with $\top$). This elimination of variables can in fact be applied not only to uniform formulas but also to formulas that are uniform {\em with respect to some variables}, so as to eliminate those `uniform' variables separately. Therefore, modulo this elimination, in the following subsections we are going to assume w.l.o.g.\ that the formulas we consider are non-uniform in each of their variables. Modulo equivalent rewriting, we can assume w.l.o.g.\ that every such formula is of the form $\varphi\rightarrow\psi$, where $\psi$ is positive, 
and all the variables occurring in $\psi$ also occur in $\varphi$.
For such formulas, we have:
\begin{center}
\begin{tabular}{r c l}
$\f, w\Vdash \varphi\to \psi$ 
& iff & $w\in \val{\varphi\rightarrow \psi}(V)$ for all $V$ on $\f$\\
& iff & for all $V$ on $\f$, if $w\in \val{\varphi}(V)$ then $w\in \val{\psi}(V).$\\
\end{tabular}
\end{center}

The textbook heuristics for producing the correspondent of formulas of this form is the `minimal valuation' method (see \cite{vB2010} subsection 9.4):  find the (class of) minimal valuation(s) $V^\ast$ on $\f$ such that $w\in \val{\varphi}(V^\ast)$ (and plug their description in the standard translation of the consequent).
The success of this heuristics rests on two conceptually different requirements: 
first, that `minimal valuations' exist; second, provided they exists, that they are tame.
The account we will present in the next sections, as described in the discussion at the end of Subsection \ref{subsec:the general reduction strategy}, will deal with these two requirements separately and in the reverse order. Namely, first we will single out subclasses of $L_0$-definable valuations (our target class of `minimal valuations'), restricting the universal quantification to which guarantees that the first-order correspondents can be effectively computed  (essentially by way of the `plug-in' method alluded to above). Second, we will show that certain order-theoretic properties of the extension maps $\val{\varphi}$, seen as operations on $\f^+$, 
guarantee that restricting the universal quantification to the target class preserves equivalence (essentially by guaranteeing the existence of a suitable `minimal valuation' taken in the target class). 
Third, we will show that the syntactic conditions on $\phi$ guarantee that $\val{\phi}$ satisfies the required order-theoretic properties.

\subsection{Very simple Sahlqvist implications}
\paragraph{The reduction strategy.} For any $m\in \mathbb{N}$, let us write $S \subseteq_m W$ if $S\subseteq W$ and $|S|\leq m$. Consider the subclass of the tame valuations which 
assign finite subsets of  bounded size $m$ to propositional variables, i.e.\ valuations $V_1:\Prop\to \pp_{m}(W)$, 
where $$\pp_{m}(W): = \{S\mid S \subseteq_m W\},$$
and suppose the following were equivalent:
\begin{enumerate}
\item  $\forall V(w\in \val{\varphi}(V) \Rightarrow w\in \val{\psi}(V))$
\item  $\forall V_1(w\in \val{\varphi}(V_1) \Rightarrow w\in \val{\psi}(V_1))$.
\end{enumerate}
This would mean that
\begin{center}
\begin{tabular}{c l}
& $\f\models \forall P_1\ldots\forall P_n ST_x(\varphi\rightarrow \psi)[x :=w ]$\\
 iff &
$\f\models \forall P_1^1\ldots\forall P_n^1 ST_x(\varphi\rightarrow
\psi)[x :=w ]$,\\
\end{tabular}
\end{center}
where the variables $P_i^1$ would not range over arbitrary subsets of $W$, but only over those of size at most $m$. Provided the equivalence between 1 and 2 above holds,  we would effectively obtain the local first-order correspondent of $\phi\to \psi$ 
by replacing each  $\forall P_i^1$ in the prefix with $\forall z^1_i \forall z^2_i \ldots \forall z^m_i$ and each atomic formula of the form $P_i^1 y$ with
$y= z^1_i \vee y= z^2_i \vee \cdots \vee y= z^m_i$, where all the $z$s are fresh variables.

\paragraph{Order-theoretic conditions.} Let $n\in\mathbb{N}$. For all $n$-tuples $\overline{X}\in \pp(W)^n$ and $\overline{m}\in \mathbb{N}^n$, 
let
\[
\sigma_{\overline{m}}(\overline{X}): =
\begin{cases}
\{\overline{Z}\mid Z_i\subseteq_{m_i} X_i\} & \mbox{if } X_i\neq \varnothing \mbox{ for each } 1\leq i\leq n\\
\varnothing & \mbox{otherwise} .\end{cases}\]   An operation $f: \pp(W)^n \rightarrow \pp(W)$ is  {\em completely} $\overline{m}$-{\em additive}\footnote{This definition is inspired to \cite[Definition 2.1(iii)]{Henkin}. When $X_i = X_j$ for all $1\leq i, j\leq n$, we will simplify notation and write $\sigma_{\overline{m}}(X)$.} if
for every $\overline{X}\in \pp(W)^n$,
\[f(\overline{X}) = \bigcup\{f(\overline{Z})\mid \overline{Z}\in \sigma_{\overline{m}}(\overline{X})\}.\]
The following lemma is an immediate consequence of the definition:
\begin{lemma}\label{Op:Composition:Props:Lemma}
If $f$ is $\overline{m}$-additive, then
\begin{enumerate}
\item $f$ is order-preserving.
\item  $f(\overline{X}) = \varnothing$ if $X_i = \varnothing$ for some $i$.
\end{enumerate}
\end{lemma}
Particularly important is the special case of $\overline{1}$-additive maps. Indeed:
\begin{lemma}
The following are equivalent for any $f: \pp(W)^n \rightarrow \pp(W)$:
\begin{enumerate}
\item $f$ is $\overline{1}$-additive;
\item $f$ is a {\em complete operator}, that is, for every $1\leq i \leq n$,
every  $\x\subseteq \pp(W)$, and all $X_1,\ldots, X_{i-1},  X_{i+1}, \ldots, X_n\in \p(W)$, 
\begin{equation}
\label{eq:completely join pres}
\begin{tabular}{c l}
& $f(X_1,\ldots, X_{i-1}, \bigcup\x, X_{i+1}, \ldots, X_n )$\\
= &  $\bigcup_{Y\in
\x} f(X_1,\ldots, X_{i-1}, Y, X_{i+1}, \ldots, X_n )$.\\
\end{tabular}
\end{equation}
\end{enumerate}
\end{lemma}
\begin{proof}
Complete operators are clearly $\overline{1}$-additive. Conversely, for every  $\overline{1}$-additive map $f$ and all  $\x$ and  $X_1,\ldots, X_{i-1},  X_{i+1}, \ldots, X_n$ as above such that $\bigcup \x\neq \varnothing\neq X_j$ for all $1\leq j\neq i\leq n$,

\begin{center}
\begin{tabular}{r c l}
$f(X_1,\ldots, X_{i-1}, \bigcup\x, X_{i+1}, \ldots, X_n )$ & = & $\bigcup\{f(\{x_1\}, \ldots,\{y\}, \ldots \{x_n\})\mid x_j\in X_j$ and  $y\in \bigcup\x\}$\\
& = & $\bigcup\{f(X_1, \ldots,Y, \ldots X_n)\mid   Y\in \x\}$.\\
\end{tabular}
\end{center}
The remaining cases are left to the reader.
\end{proof}
Consider the following conditions on $\phi$:
\begin{enumerate}
\item[(a)]$\phi(p_1, \ldots, p_n)= \phi'(p_1, \ldots, p_n, \gamma_1, \ldots, \gamma_{\ell})$, where there are $m_i$ occurrences of $p_i$ in $\phi'$ for each $1\leq i\leq n$;
\item[(b)] for any $\overline{A}\in \pp(W)^\ell$, the extension map $\val{\phi'(p_1, \ldots, p_n, \overline{A})}$ is an $\overline{m}$-additive operation on $\pp(W)$, where each coordinate in $\overline{m}$ is defined as in item (a);
\item[(c)] $\val{\gamma_1}$ to $\val{\gamma_{\ell}}$ are order-reversing in each coordinate.
\end{enumerate}

For every frame $\f$ and every $m\in \mathbb{N}$, let $\mathsf{Val}_1(\f)$ be the set of valuations on $\f$ of type $V_1: \mathsf{Prop}\to\p_m(W)$.
\begin{prop}
\label{prop:very simple}
Let $\varphi\to \psi\in \ML$ be such that $\phi$ verifies the conditions (a)-(c) above and $\val{\psi}$ is order-preserving. Let $m = max_{i=1}^n m_i$. 
Then the
following are equivalent for every frame $\f$:
\begin{enumerate}
\item  $(\forall V\in \mathsf{Val}(\f))[w\in \val{\varphi}(V) \Rightarrow w\in \val{\psi}(V)]$
\item  $(\forall V_1\in \mathsf{Val}_{1}(\f))[w\in \val{\varphi}(V_1) \Rightarrow w\in \val{\psi}(V_1)].$
\end{enumerate}
\end{prop}
\begin{proof}
(1 $\Rightarrow$ 2) Clear. (2 $\Rightarrow$ 1) 
%
Let  $V\in \mathsf{Val}(\f)$ and  $w\in W$ s.t.\ $w\in \val{\varphi}(V)$. Hence, $$\emptyset
\neq \val{\varphi}(V) = \val{\varphi'}(V(p_1),\ldots V(p_n), \val{\gamma_1}(V), \ldots, \val{\gamma_{\ell}}(V)).$$ By
Lemma \ref{Op:Composition:Props:Lemma}(2), this implies that $V(p_i)\neq
\emptyset$ for every $1\leq i \leq n$. 
By assumption (b),
$$\val{\varphi}(V)= \bigcup\{\val{\varphi'}(S_1, \ldots,
S_n, \val{\gamma_1}(V), \ldots, \val{\gamma_{\ell}}(V))\mid S_i\subseteq_{m_i}V(p_i),\  1\leq i\leq n\}.$$
Hence, $w\in \val{\varphi}(V)$ implies that
$w\in \val{\varphi'}(T_1, \ldots,
T_n, \val{\gamma_1}(V), \ldots, \val{\gamma_{\ell}}(V))$ for some $ T_i\subseteq_{m_i}V(p_i),$ $1\leq i \leq n$.
Let $V_1$ be the valuation that maps any $q\in \Prop\setminus\{p_i\mid 1\leq i\leq n\}$ to $\emptyset$ and
$p_i$ to $T_i$ for each $1\leq i\leq n$. Clearly, $V_1\in \mathsf{Val}_1(\f)$.  Moreover, $w\in
\val{\varphi}(V_1)$. Indeed,
\begin{center}
\begin{tabular}{r c l}
$w$ & $\in$ & $\val{\varphi'}(T_1, \ldots,
T_n, \val{\gamma_1}(V), \ldots, \val{\gamma_{\ell}}(V))$\\
& $\subseteq$ & $\val{\varphi'}(T_1, \ldots,
T_n, \val{\gamma_1}(V_1), \ldots, \val{\gamma_{\ell}}(V_1))$\\
& $=$ & $\val{\varphi'}(V_1(p_1), \ldots,
V_1(p_n), \val{\gamma_1}(V_1), \ldots, \val{\gamma_{\ell}}(V_1))$\\
& $ = $ & $\val{\varphi}(V_1).$\\
\end{tabular}
\end{center}
The inclusion in the chain above holds since $V_1(p)\subseteq V(p)$ for every $p\in \mathsf{Prop}$ and the extensions of the $\gamma$s are $\subseteq$-reversing by assumption (c). Hence, by assumption (2), $w\in \val{\psi}(V_1)$.
Since $\val{\psi}$ is order-preserving in every coordinate, and again $V_1(p)\subseteq V(p)$ for every $p\in \mathsf{Prop}$,
we get $w\in \val{\psi}(V_1)\subseteq \val{\psi}(V)$, which concludes the
proof.
\end{proof}

\paragraph{Syntactic conditions.} It remains to show that the very simple Sahlqvist implications verify the assumptions of Proposition \ref{prop:very simple}. The assumptions on $\psi$ and the $\gamma$s are verified because of Proposition \ref{orderpreserving operations}. As to the assumptions on $\phi'$: 
\begin{lemma}
\label{very simple->complete operator}
If $\phi'(p_1,\ldots, p_n, \overline{c})$ is a formula built up from  variables $\overline{p}$ and parameters $\overline{c}$ using $\Diamond$ and $\wedge$, then $\val{\phi'}: \pp(W)^n\to \pp(W)$ is an $\overline{m}$-additive map, where  $m_i$ is the number of occurrences of $p_i$ in $\phi'$  for each $1\leq i\leq n$.
\end{lemma}

\begin{proof}
The proof is by induction on $\phi'$. The base cases for $\phi'$ a propositional variable $p_i$ or a parameter $c$ are immediate. The inductive step for $\phi'$ of the form $\Diamond \psi$ is straightforward, so we only consider the case for $\phi'$ of the form $\psi \wedge \theta$. Without loss of generality, we can restrict our attention to the $i$th variable. By assumption $p_i$ occurs $m_i$ times in $\psi \wedge \theta$, so there are $\ell, k \in \mathbb{N}$ such that $\ell + k = m_i$ and $p_i$ occurs $\ell$ times in $\psi$ and $k$ times in $\theta$.
\begin{eqnarray*}
\val{\phi'}(X) &= &\val{\psi}(X) \cap \val{\theta}(X)\\
&= &\bigcup \{ \val{\psi}(X') \mid X' \subseteq_{\ell} X\} \cap \bigcup \{ \val{\theta}(X'') \mid X'' \subseteq_{k} X\}\\
&= &\bigcup \{ \val{\psi}(X') \cap \val{\theta}(X'') \mid X' \subseteq_{\ell} X,  X'' \subseteq_{k} X\}\\
&= &\bigcup \left\{\left. \bigcup \{ \val{\psi}(X') \cap \val{\theta}(X'') \mid X' \subseteq_{\ell} X''',  X'' \subseteq_{k} X'''\} \right| X''' \subseteq_{m_i} X \right\}\\
&= &\bigcup \left\{\left. \bigcup \{ \val{\psi}(X') \mid X' \subseteq_{\ell} X''' \} \cap \bigcup \{\val{\theta}(X'') \mid  X'' \subseteq_{k} X'''\} \right| X''' \subseteq_{m_i} X \right\}\\
&= &\bigcup \left\{\left. \val{\psi}(X''') \cap \val{\theta}(X''') \right| X''' \subseteq_{m_i} X \right\}\\
&= &\bigcup \left\{\left. \val{\phi'}(X''') \right| X''' \subseteq_{m_i} X \right\}.
\end{eqnarray*}
\end{proof}
\begin{cor}
\label{cor:very simple->complete operator} If $\varphi = \varphi(p_1,\ldots, p_n)$ is a very simple Sahlqvist antecedent then it verifies the assumptions (a)-(c) of Proposition \ref{prop:very simple}. In particular, the maps $\val{\gamma}$s are exactly the ones induced by the negative formulas occurring in the construction of $\phi$, the map $\val{\varphi'}$ is induced by the compound occurrences of $\wedge$ and $\Diamond$, and  $m_i$ is the number of positive  occurrences of $p_i$ in $\varphi$ for every $1\leq i\leq n$.
\end{cor}
%
%
%
%
\begin{example}\label{Simple:Sahl:Example}
Let us consider the very simple Sahlqvist formula
\[
p\wedge \Diamond  p\to \Box  p,
\]
%
which locally corresponds to the property of having at most one $R$-successor\footnote{This condition is also definable by the more familiar $\Diamond  p\to \Box  p$.}, i.e.\
\[
\forall z\forall u(Rxz \wedge Rxu \rightarrow z = u).
\]
The variable $p$ occurs twice positively in the antecedent, making $\val{p\wedge \Diamond  p}$ a 2-additive map. Hence, according to our reduction strategy, the monadic second-order quantification in the second-order translation
\[
\forall P [P(x) \wedge \exists y (Rxy \wedge P(y)) \to \forall u(Rxu \to P(u))]
\]
can be equivalently restricted to subsets of size at most 2. Doing this yields the equivalent $L_0$-formula
\begin{center}
$\forall z_1 \forall z_2[(x = z_1\vee x = z_2) \wedge \exists y (Rxy \wedge (y = z_1 \vee y = z_2)) \to \forall u(Rxu \to (u = z_1 \vee u = z_2))]$.
\end{center}
This can be simplified to
\begin{center}
$\forall z_1 \forall z_2[(x = z_1 \vee x = z_2) \wedge (Rxz_1\vee Rxz_2)) \to \forall u(xRu \to (u = z_1 \vee u = z_2))]$,
\end{center}
and reasoning a bit further this can be seen to be equivalent to
\begin{center}
$\forall z_1 \forall z_2[(Rxz_1\vee Rxz_2)) \to \forall u(Rxu \to (u = z_1 \vee u = z_2))]$,
\end{center}
which, in turn, is equivalent to $\forall z\forall u(Rxz \wedge Rxu \rightarrow z = u)$.

As seen above, the reduction strategy does not immediately yield the simplest first-order equivalent possible. Some further simplification will usually be possible, as will also be seen in examples \ref{Sahl:Example} and \ref{Atom:Ind:Example}. More optimal equivalents could be produced at the cost of complicating the reduction strategy. This will be further discussed in the conclusion.
\end{example}

\subsection{Sahlqvist implications}
\label{ssec:sahlqvist:implications}

\paragraph{The reduction strategy.}
Another promising subclass of tame valuations is formed by those $V_2\in \mathsf{Val}(\f)$ such that 
for
every $p\in \Prop$,  $V_2(p)= R[z]$ for some $z\in W$.
Indeed, suppose that the following were
equivalent:
\begin{enumerate}
\item  $\forall V(w\in \val{\varphi}(V) \Rightarrow w\in \val{\psi}(V))$
\item  $\forall V_2(w\in \val{\varphi}(V_2) \Rightarrow w\in \val{\psi}(V_2))$.
\end{enumerate}
This would mean that
$$\f\models \forall P_1\ldots\forall P_n ST_x(\varphi\rightarrow \psi)[x := w]\quad
\mbox{ iff }\quad
\f\models \forall P_1^2\ldots\forall P_n^2 ST_x(\varphi\rightarrow
\psi)[x := w]$$ where the variables $P_i^2$ would not range over
$\pp(W)$, but only over $\{R[z]\ |\ z\in W\}$.
Therefore the
formula on the right-hand side of the `iff' above can be transformed into a
first-order formula 
by replacing each  $\forall P_i^2$ in the prefix with $\forall z_i$
and each atomic formula of the form $P_i^2 y$ with $Rz_i y$.

Actually, this argument can be extended to valuations $V_2$ such that for every $p\in \Prop$,
$V_2(p) = \bigcup R^{k_j}[z_j]$ for a finite set $\{z_1,\ldots z_m\}\subseteq  W$ of uniformly bounded size $m$,
and some $k_j\in \mathbb{N}$  bounded by the {\em modal depth} of the formula, i.e.\ the maximum number of nested boxes occurring in the formula.  
Notice that the valuations $V_1$ ranging over subsets of bounded size are the
special case of $V_2$ where all $k_j$ are equal to $0$. 
The formula on the right-hand side of the `iff' above can be
equivalently transformed into a first-order formula
by replacing each  $\forall P^2$ in the prefix with $\forall z_{1}\cdots\forall z_{m}$ (choosing fresh individual variables for each $P$),
and each formula of the form $P^2 y$ with
an $L_0$-formula which says `there exists an $R$-path from one of the   $z_j$s to $y$ in $k_j$ steps'. Such formula is a disjunction for $j\in \{1,\ldots,m\}$ of formulas of the following form:
\[
\exists v_0,\ldots v_{k_j}[z_j= v_0 \wedge \bigwedge_{\ell = 0}^{k_j - 1} R v_\ell  v_{\ell+1}  \wedge v_{k_j} = y].
\]
%
This time we are after some conditions on $\varphi$ and $\psi$
that guarantee that the universal quantification  $\forall V$ can be equivalently
replaced with the universal quantification  $\forall V_2$.
\paragraph{Order-theoretic conditions.}
The maps $f, g: \pp(W)\rightarrow \pp(W)$ form
an {\em adjoint pair} (notation: $f\dashv g$) iff for every $X, Y\in \pp(W)$,
\[
f(X)\subseteq Y\ \mbox{ iff }\   X\subseteq g(Y).
\]
Whenever $f\dashv g$, $f$ is the {\em left adjoint} of $g$ and $g$
is {\em the right adjoint} of $f$. One important property of adjoint
pairs of maps is that if a map  admits a
left (resp.\ right) adjoint, the adjoint is unique and can be
computed pointwise from the map itself and the order (which in our case is the inclusion). This means that admitting
a left (resp.\ right) adjoint is an intrinsically
order-theoretic property of maps.

\begin{prop}
\label{prop:right adjoints}
\begin{enumerate}
\item Right adjoints between complete lattices are exactly the {\em completely meet-preserving maps}, i.e.\ in the concrete case of powerset algebras $\p(W)$ right adjoints are exactly those maps $g$ such that $g(\bigcap \x) = \bigcap \{g(X) \mid X \in \x \}$ for all $\x \subseteq \p(W)$.
\item Right adjoints on a powerset algebra $\p(W)$ are exactly maps of the form $l_{\s}$ (defined in \eqref{def:lr}) for some binary relation $\s$ on $W$.
\item For any binary relation $\s$ on $W$, the left adjoint of $l_{\s}$ is the map $m_{\s^{-1}}$, defined by the assignment $X\mapsto \s[X]$.
\end{enumerate}
\end{prop}

\begin{proof}
1. See \cite[Proposition 7.34]{DaPr}.
2.  We leave to the reader to verify that every map of form $l_{\s}$  is completely meet-preserving, hence is a right adjoint. Conversely, let $g:
\pp(W)\to \pp(W)$ be  a right adjoint. Then by item 1 above, $g$ is completely meet-preserving. Define
$\s\subseteq W\times W$ as follows: for every $x,z\in W$,

\begin{center}
$ \s xz \quad $ iff $\quad x\not\in   g(W\setminus \{z\}).$
\end{center}
Hence,
\begin{center}
 $x\in l_\s(W\setminus\{z\})\ $  iff  $\ \s[x]\subseteq (W\setminus\{z\})\ $ iff $\
z\notin \s[x]\ $ iff
 $\ x\in g(W\setminus\{z\}),$
\end{center}
which shows our claim for all the special subsets of $W$ of type
$W\setminus\{z\}.$ In order to show it in general, fix $X\in \pp(W)$
and notice that $X= \bigcap_{z\notin X} (W\setminus\{z\})$. Using
the fact that $g$ is completely meet-preserving and the
special case shown above, we get:

\begin{center}
\begin{tabular}{r c l l}
$g(X)$&$ =$& $ g(\bigcap\{
(W\setminus\{z\})\ |\ z\notin X\})$& \\
&$=$& $\bigcap\{g(W\setminus\{z\})\ |\ z\notin X\}$ & \\
&$=$& $\bigcap\{l_\s(W\setminus\{z\})\ |\ z\notin X\} $& \\
&$=$& $l_\s(\bigcap\{(W\setminus\{z\})\ |\ z\notin X\}$& ($\ast$)\\
&$=$& $l_\s(X).$& \\
\end{tabular}
\end{center}
The marked equality can be verified directly, but also follows
from the more general fact that $l_{\s}$ is completely meet-preserving for every $\s$. 3. Left to the reader.
\end{proof}
\commment{
The reason why we are interested in adjoint pairs of maps is that,
for every frame $\f= (W, R)$, $m_{R^{-1}}\dashv l_{R}$ (and
$m_{R}\dashv l_{R^{-1}}$), i.e.\ for every $X, Y\in \pp(W)$,
\[
m_{R^{-1}}(X)\subseteq Y\ \mbox{ iff }\  X\subseteq l_R(Y).
\]
}
%
Consider the following conditions on $\phi$:
\begin{enumerate}
\item[(a)] $\varphi(p_1,\ldots, p_n) = \varphi'(\chi_1(p_{j_1})/q_1,\ldots, \chi_{n'}(p_{j_{n'}})/q_{n'}, \gamma_1, \ldots, \gamma_{\ell})$, where each placeholder variable $q_1,\ldots, q_{n'}$ occurs exactly once in $\phi'$,\footnote{This notation implies that $n'\geq n$, and hence that multiple occurrences are possible in $\phi$ of each variable $p_1,\ldots, p_n$.} and each $\chi$-formula contains exactly one variable $p\in \{p_1,\ldots, p_n\}$, and moreover 
\item[(b)] for any $\ell$-tuple  of parameters $\overline{A}$, the meaning function $\val{\varphi'(\overline{q}, \overline{A})}$ is a $\overline{1}$-additive map;
%
\item[(c)]   $\val{\chi_i}: \pp(W)\to \pp(W)$ is a right
adjoint for each $1\leq i\leq n'$,
i.e.\ there exists some $f_i: \pp(W)\to \pp(W)$ such that $f_i\dashv\val{\chi_i}$;
%
\item[(d)] for every $1\leq i\leq n'$, the map $f_i$  is defined by the assignment $X\mapsto R^{k_i}[X]$; 
\item[(e)] $\val{\gamma_1}$ to $\val{\gamma_{\ell}}$ are order-reversing in each coordinate.
\end{enumerate}

Notice that, by Proposition \ref{prop:right adjoints}, condition (c) already guarantees that for every $i$, $f_i$  is defined by $X\mapsto \s_i[X]$ for some {\em arbitrary} binary relation $\s_i$ on $W$; however, since $\s_i$ is arbitrary,  this is not yet enough to guarantee that valuations defined by $p_i\mapsto f_i(X_i)$ for some finite $X_i\subseteq W$  be $L_0$-definable. Condition (d) above guarantees this last point.

Notice that if $\chi_i(p_{j_i}) = p_{j_i}$ for every $1\leq i\leq n'$, then the conditions above become an equivalent, if notationally heavier, version of conditions (a)-(c) of the previous subsection.

For every frame $\f$, 
let $\mathsf{Val}_2(\f)$ be the set of valuations on $\f$ such that, for every $p\in \Prop$, $V_2(p) = \bigcup R^{k_i}[z_i]$ for some $\{z_1,\ldots z_m\}\subseteq W$ and some $k_i\in \mathbb{N}$ bounded by the modal depth of $\phi$, where $m$ is the maximum number of times that a variable occurs in $\phi$. 

%
%
\begin{prop}
\label{semantic on simple sahl} Let $\varphi\to \psi\in \ML$ be such that $\varphi$
verifies conditions (a)-(e) above and $\val{\psi}$ is order-preserving in each coordinate. 
Then the following are equivalent:
\begin{enumerate}
\item  $(\forall V\in \mathsf{Val}(\f))[w\in \val{\varphi}(V) \Rightarrow w\in \val{\psi}(V)]$
\item  $(\forall V_2\in \mathsf{Val}_2(\f))[w\in \val{\varphi}(V_2) \Rightarrow w\in \val{\psi}(V_2)]$.
\end{enumerate}
\end{prop}
\begin{proof} (1 $\Rightarrow$ 2) Clear. (2 $\Rightarrow$ 1) 
Let  $V\in \mathsf{Val}(\f)$ and  $w\in W$ s.t.\ $w\in \val{\varphi}(V)$. Hence, $$\emptyset
\neq \val{\varphi}(V)= \val{\varphi'}(\val{\chi_1}(V(p_{j_1})),\ldots,
\val{\chi_{n'}}(V(p_{j_{n'}})), \val{\gamma_1}(V), \ldots, \val{\gamma_{\ell}}(V)).$$ By assumption (b) and  Lemma \ref{Op:Composition:Props:Lemma}, this implies that 
$\val{\chi_i}(V(p_{i}))\neq \emptyset$ for every  $1\leq i\leq n'$, and
moreover, 
the following chain of equalities holds:
\begin{center}
\begin{tabular}{c l l}
& $w\in \val{\varphi}(V)$\\
$=$& $ \val{\varphi'}(\val{\chi_1}(V(p_{j_1})),\ldots,
\val{\chi_{n'}}(V(p_{j_{n'}})), \val{\gamma_1}(V), \ldots, \val{\gamma_{\ell}}(V))$ & assumption (a)\\
$=$ & $\bigcup\{\val{\varphi'}(\{x_1\},\ldots, \{x_{n'}\}, \val{\gamma_1}(V), \ldots, \val{\gamma_{\ell}}(V))\ |\ [x_i\in
\val{\chi_i}(V(p_{j_i}))]_{i=1}^{n'}\}$ & assumption (b)\\
$=$ & $\bigcup\{\val{\varphi'}(\{x_1\},\ldots, \{x_{n'}\}, \val{\gamma_1}(V), \ldots, \val{\gamma_{\ell}}(V))\ |\ [f_i(\{x_i\})\subseteq
V(p_{j_i})]_{i=1}^{n'}\}$ & assumption (c)\\
%
\end{tabular}
\end{center}
%
Then $w \in \val{\varphi'}(\{z_1\},\ldots, \{z_{n'}\}, \val{\gamma_1}(V), \ldots, \val{\gamma_{\ell}}(V)))$ for some
 $z_1,\ldots,z_{n'}\in W$ such that $f_i(\{z_i\})\subseteq V(p_{j_i})$ for each $1\leq i\leq n'$.
Let $V_2$ be the valuation that maps any $q\in \Prop\setminus\{p_1,\ldots, p_n\}$ to $\emptyset$, and each $p\in \{p_1,\ldots,p_n\}$ to $\bigcup\{f_j(\{z_j\})\mid f_j(\{z_j\})\subseteq V(p)\}$.
By assumption (d), $V_2\in \mathsf{Val}_2(\f)$. Moreover $f_i(\{z_i\})\subseteq V(p_{j_i})$ for every $1\leq i\leq n'$.
Let us show that
$w \in \val{\varphi}(V_2)$:
\begin{center}
\begin{tabular}{r  l l}
 $w\in $& $\val{\varphi'}(\{z_1\},\ldots, \{z_{n'}\}, \val{\gamma_1}(V), \ldots, \val{\gamma_{\ell}}(V))$\\
 $\subseteq $& $\bigcup\{\val{\varphi'}(\{x_1\},\ldots, \{x_{n'}\}, \val{\gamma_1}(V), \ldots, \val{\gamma_{\ell}}(V))\ |\
[f_i(\{x_i\})\subseteq V_2(p_{j_i})]_{i=1}^{n'}\}$\\
 $\subseteq $& $\bigcup\{\val{\varphi'}(\{x_1\},\ldots, \{x_{n'}\}, \val{\gamma_1}(V_2), \ldots, \val{\gamma_{\ell}}(V_2))\ |\
[f_i(\{x_i\})\subseteq V_2(p_{j_i})]_{i=1}^{n'}\}$ & assumption (e)\\
 $= $& $\bigcup\{\val{\varphi'}(\{x_1\},\ldots, \{x_{n'}\}, \val{\gamma_1}(V_2), \ldots, \val{\gamma_{\ell}}(V_2))\ |\
[x_i\in \val{\chi_i}(V_2(p_{j_i}))]_{i=1}^{n'}\}$ &  assumption (c)\\

%
%
 $=$ & $ \val{\varphi'}(\val{\chi_1}(V_2(p_{j_1})),\ldots, \val{\chi_{n'}}(V_2(p_{j_{n'}})), \val{\gamma_1}(V_2), \ldots, \val{\gamma_{\ell}}(V_2))$ & assumption (b)\\
 $=$ & $ \val{\varphi}(V_2)$.
\end{tabular}
\end{center}
By assumption (2), we can conclude that $w\in \val{\psi}(V_2)$. Since
$\val{\psi}$ is order-preserving in every coordinate, and
and $V_2(p) = \bigcup\{f_j(\{z_j\})\mid f_j(\{z_j\})\subseteq V(p)\} \subseteq V(p)$ for every  $p\in \{p_1,\ldots ,p_n\}$, 
%
%
this implies that $w\in \val{\psi}(V_2)\subseteq \val{\psi}(V)$, as required.
\end{proof}

\paragraph{Syntactic conditions.}

\begin{prop}
\label{prop:simple sahlqvist syntactic cond}
Any Sahlqvist implication $\varphi\to \psi\in \ML$  
verifies the assumptions of Proposition \ref{semantic on
simple sahl}. In particular, the maps $\val{\chi_i}$ are exactly those induced by the boxed atoms. 
\end{prop}
\begin{proof}
The statement follows by Lemma 
\ref{very simple->complete operator} and Proposition \ref{prop:right adjoints},
and the additional fact that,  $l_{R_2}\circ
l_{R_1} = l_{R_1\circ R_2}$ for every $R_1, R_2\subseteq W\times W$. 
 \end{proof}

\begin{example}\label{Sahl:Example}
Consider the (definite) Sahlqvist implication $\Diamond \Box p \wedge \Box q \rightarrow \Box \Diamond (p \wedge q)$. This has standard second-order frame equivalent
\[
\forall P \forall Q [(\exists y (Rxy \wedge \forall u (Ryu \rightarrow P(u))) \wedge \forall v (Rxv \rightarrow Q(v))) \rightarrow \forall w(Rxw \rightarrow \exists s (Rws \wedge P(s) \wedge Q(s)))].
\]
The reduction strategy prescribes that in the above we replace $\forall P \forall Q$ with $\forall z_1 \forall z_2$, and substitute $P(y)$ and $Q(y)$ with $\exists v_0 \exists v_1 (v_0 = z_1 \wedge R v_0  v_1 \wedge v_1 = y)$ and $\exists v_0 \exists v_1 (v_0 = z_2 \wedge Rv_0  v_1 \wedge v_1 = y)$, respectively, which simplify to $Rz_1  y$ and $Rz_2  y$, respectively. Doing this we obtain the first-order frame equivalent
\[
\forall z_1 \forall z_2 [(\exists y (Rxy \wedge \forall u (Ryu \rightarrow Rz_1  u)) \wedge \forall v (Rxv \rightarrow Rz_2  v)) \rightarrow \forall w(Rxw \rightarrow \exists s (Rws \wedge z_1Rs \wedge z_2Rs))].
\]
Using the well-known fact that for any first-order formula $\beta(x,y)$ it holds that $\forall x \forall y \beta(x,y) \models \forall x \forall x \beta(x,x)$, we see (by pulling out quantifiers and setting $z_1 = y$ and $z_2 = x$) that the above has as consequence
\[
\forall y \forall w (xRy \wedge xRw \rightarrow \exists s(Rxs \wedge Rys \wedge Rws)).
\]
An easy semantic argument shows that the converse also holds, and hence that the last formula is actually a local first-order frame correspondent for $\Diamond \Box p \wedge \Box q \rightarrow \Box \Diamond (p \wedge q)$.
\end{example}

\subsection{Atomic inductive formulas}\label{Atomic:Ind:Section}

\paragraph{The reduction strategy.} In the case of Sahlqvist implications considered in the previous subsection, the tame valuations of the class $\mathsf{Val}_2$ mapped propositional variables to unions of direct images of singletons under definable {\em binary} relations. In the present subsection we would like to generalize this by defining a class of tame valuations mapping propositional variables to unions of direct images of singletons under  definable relations {\em of any arity} $n\geq 2$. Taking the direct image under an $n$-ary relation requires $n-1$ inputs, so apart from the singletons there will have to be other parameters which need to be tame subsets themselves.  This suggest the idea to consider a class of tame valuations in which each valuation is defined by induction on some well founded order $\Omega$ on atomic propositions. The base of the induction will define tame assignments on the $\Omega$-minimal propositional variables  as in valuations of type $\mathsf{Val}_2$. For the inductive step,  the previously defined variable assignments are used as parameters in the definition of tame assignments for non $\Omega$-minimal propositional variables, using definable relations of arity $n\geq 3$. Working out this idea will require a few definitions and some notation for bookkeeping. We start by introducing these.

For every $i+1$-ary relation $\s_i$ on $W$ and all $X_1,\ldots, X_i\subseteq W$, let

\[\s_i[X_1,\ldots, X_i] := \{y\in W\mid \exists x_1\cdots\exists x_i[\bigwedge_{h =1}^i x_h\in X_h\ \wedge\ \s_i(x_1,\ldots, x_i, y)]\}.\]

\noindent We will abuse notation an write $\s_i[z, X_2,\ldots, X_i]$ for $\s_i[\{z\}, X_2,\ldots, X_i]$. For each $n \in \mathbb{N}^+$, let $\mathsf{SubSeq}(n)$ denote the set of subsequences of the sequence $p_1, p_2, \ldots, p_n$ of propositional variables, ordered by $p_i\Omega p_j$ iff $i\leq j$. Let $\mathsf{SubSeq}(0)$ be the singleton containing the empty sequence $\bm{\epsilon}$. We will denote elements of $\mathsf{SubSeq}(n)$ by ${\bm \rho}$, write $\ell({\bm \rho})$ for the length of ${\bm \rho}$, and ${\bm \rho}^i$ for the $i$-th term of $\bm{\rho}$. Suppose we are given a finite set of triples $Q_i \subseteq W \times \mathsf{SubSeq}(i - 1) \times \mathbb{N}$, for each $1 \leq i \leq n$. Notice that for $i = 1$ the elements of $Q_1$ are of the form $(w,{\bm \epsilon}, k)$. 
We define $V_3$ inductively as follows:
\begin{enumerate}
\item $V_3(p_1) = \bigcup_{(w,{\bm \epsilon}, k) \in Q_1} \s_{{\bm \epsilon}}^k [w]$, where $\s_{{\bm \epsilon}}^k = R^{k}$ for each $k$; 
\item for every $1 < i \leq n$, $V_3(p_i) = \bigcup_{(w,\bm{\rho}, k) \in Q_i} \s_{{\bm \rho}}^{k}[w, V_3(\bm{\rho}^1),\ldots, V_3(\bm{\rho}^{\ell(\bm{\rho})})]$, where $\s_{\bm{\rho}}^{k}$ is such that for all $x_0,\ldots, x_{\ell(\bm{\rho})}, y\in W$,
\begin{equation}
\label{eq:L_0 definable i+1 ary rel}
\s_{\bm{\rho}}^{k}(x_0,\ldots, x_{\ell(\bm{\rho})}, y)\
\mbox{ iff }\  (\bigwedge_{0 \leq h < \ell(\bm{\rho})} R x_h x_{h+1}) \wedge R^{k} x_{\ell(\bm{\rho})}  y.
\end{equation}

\end{enumerate}
%
Notice that when $n=1$, the valuations $V_3$ reduce to special valuations $V_2$ of Section \ref{ssec:sahlqvist:implications}.
Suppose that the following were equivalent:
\begin{enumerate}
\item  $\forall V(w\in \val{\varphi}(V) \Rightarrow w\in \val{\psi}(V))$
\item  $\forall V_3(w\in \val{\varphi}(V_3) \Rightarrow w\in \val{\psi}(V_3))$.
\end{enumerate}
This would mean that
\begin{equation}\label{V3:Equation}
\f\models \forall P_1\ldots\forall P_n ST_x(\varphi\rightarrow \psi)[x: = w]\quad
\mbox{ iff }\quad
\f\models \forall P_1^3\ldots\forall P_n^3 ST_x(\varphi\rightarrow
\psi)[x: = w],
\end{equation}
where the variables $P_i^3$ would not range arbitrarily on $\pp(W)$, but only on sets as described in the first enumeration above. In this case, the right-hand side of (\ref{V3:Equation}) can be equivalently transformed into a first-order formula by the following procedure, that we define inductively:
%
\begin{enumerate}
\item Replace $\forall P_1^3$ in the prefix with a string of universally quantified fresh variable $\forall z_{{\bm q}}$ for ${\bm q} \in Q_1$, and substitute each subformula of the form $P_1^3 y$ with the  formula $\alpha_{p_1}(y)$ defined to be the disjunction of the formulas $R^{k}z_{\bm q}  y$ for each ${\bm q} = (w, \epsilon, k) \in Q_1$. 
%
\item Suppose that, for each $1\leq h<i$, $\forall P_h^3$ in the prefix has been replaced by a string of universally quantified fresh individual variables, and in the matrix each subformula of the form $P_h^3 y$ has been substituted with the $L_0$-formula $\alpha_{p_h}(y)$. Then, replace $\forall P_i^3$ in the prefix with a string of universally quantified fresh variable $\forall z_{{\bm q}}$ for ${\bm q} \in Q_i$, and substitute each subformula of the form $P_i^3 y$ with the formula $\alpha_{p_i}(y)$ defined to be the disjunction of the following formulas, one for each ${\bm q} = (w, \bm{\rho}, k) \in Q_i$:
    \[
    \exists v_0,\ldots \exists v_{\ell(\bm{\rho})}[z_{\bm q} = v_0 \wedge (\bigwedge_{j = 0}^{\ell(\bm{\rho}) - 1} Rv_j  v_{j+1} \wedge \alpha_{{\bm \rho}_{j+1}}(v_{j+1})) \wedge R^{k}v_{\ell(\bm{\rho})}  y].
    \]
    %
\end{enumerate}

This time we are after some conditions on $\varphi$ and $\psi$ that guarantee that the universal quantification $\forall V$ can be replaced with the universal quantification  $\forall V_3$.

\paragraph{Order-theoretic conditions.}
The notion of adjunction of monotone maps can be generalized to $j$-ary maps in a component-wise fashion:
a $j$-ary map $f: \p(W)^j \to \p(W)$ is {\em residuated} if there exists a collection of maps
\[
\{g_h: \p(W)^j\to \p(W)\mid 1\leq h \leq j\}
\]
s.t.\ for every $1\leq h\leq j$ and for  all $X_1,\ldots, X_j, Y\in \p(W)$,
\[
f(X_1,\ldots, X_j)\subseteq Y\quad \mbox{ iff }\quad  X_h\subseteq g_h(X_1,\ldots, X_{h-1},Y,X_{h+1},\ldots, X_j).
\]
The map $g_h$ is the {\em $h$-th residual} of $f$. The facts stated in the following example and proposition are well known in the literature in their binary instance (cf.\ \cite[Subsection 3.1.3]{GJKO07}):

\begin{example}
For every $j+1$-ary relation $\s$ on $W$ and every $(X_1,\ldots, X_j)\in \p(W)^j$, let
\[\s_i[X_1,\ldots, X_j] := \{y\in W\mid \exists x_1\cdots\exists x_i[\bigwedge_{h =1}^j x_h\in X_h\ \wedge\ \s_j(x_1,\ldots, x_j, y)]\}.\]
    %
The $j$-ary operation on $\p(W)$ defined by the assignment
\begin{equation} \label{eq: residuated on P}
(X_1,\ldots, X_j)\mapsto \s[X_1,\ldots, X_j]
\end{equation}
is residuated and its $h$-th residual is the map $$g_h: \p(W)^j\to \p(W)$$ which maps every $j$-tuple  $(X_1,\ldots,X_{h-1}, Y, X_{h+1},\ldots, X_j)$  to the set
\[
\{w\in W\mid \alpha_{\s}^h(w)\},
\]
where $\alpha_{\s}^h(w)$ is the following formula:
    \begin{center}
    $\forall x_1\cdots\forall x_j \forall y[(\bigwedge_{k\neq h} x_k\in X_k\ \land \ \s(x_1,\ldots, w,\ldots, x_j, y))\Rightarrow y\in Y].$
    \end{center}
\end{example}
In the next proposition we collect a few of the many interesting  properties of residuated maps which are relevant for our subsequent exposition.

\begin{prop}
\label{prop:residuated maps}
If $f: \p(W)^j\to \p(W)$ is residuated and $\{g_h: \p(W)^j\to \p(W)\mid 1\leq h\leq j\}$ is the collection of its residuals, then:
\begin{enumerate}
\item $f$ is order-preserving in each coordinate, and for $1\leq h\leq j$, $g_h$ is order-preserving in its $h$-th coordinate;
\item $f$  preserves arbitrary joins in each coordinate; 
\item $f$ coincides with the map defined by the assignment (\ref{eq: residuated on P}), for some $j+1$-ary relation $\s$ on $W$.
%
\end{enumerate}
\end{prop}

\begin{proof}
1. Fix $1\leq h\leq j$, let $X_1,\ldots, X_j, Y, Z \in\p( W)$, and assume that $Y\subseteq Z$. By residuation, the ``tautological'' inclusion
$$f(X_1,\ldots, Z,\ldots, X_j)\subseteq f(X_1,\ldots, Z,\ldots, X_j)$$
is equivalent to the second inclusion in the following chain:
$$Y\subseteq Z\subseteq g_h(X_1,\ldots, f(X_1,\ldots, Z,\ldots, X_j), \ldots, X_j),$$
which yields, again by residuation,
$$f(X_1,\ldots, Y,\ldots, X_j)\subseteq f(X_1,\ldots, Z,\ldots, X_j).$$
The proof that $g_h$ is monotone in the $h$-th coordinate goes likewise.

\commment{
Fix $1\leq k\leq i$, s.t.\ $k\neq h$, let $X_1,\ldots,X_{k-1}, Y, Z,X_{k+1},\ldots, X_i\in\p( W)$, and assume that $Y\subseteq Z$. The following chain of inclusion holds\footnote{For sake of readability, we omit idle coordinates and represent $i$-tuples as $(\ldots[X_h,X_k]\ldots)$, with the active coordinates in display, of course ignoring the order in which they might occur in the tuple.}:
\begin{center}
\begin{tabular}{c l}
& $g_h(\ldots[X_h,Z]\ldots)$\\
$\subseteq$& $g_h(\ldots[f(\ldots[g_h(\ldots[X_h,Z]\ldots), Y]\ldots),Y]\ldots)$\\
$\subseteq$ &
  $g_h(\ldots[f(\ldots[g_h(\ldots[X_h,Z]\ldots), Z]\ldots),Y]\ldots)$\\
$\subseteq$& $ g_h(\ldots[X_h,Y]\ldots ).$\\
\end{tabular}
\end{center}
The first inclusion is obtained by applying residuation to the ``tautological'' inclusion $$f(\ldots[g_h(\ldots[X_h, Z]\ldots), Y]\ldots)\subseteq f(\ldots[
g_h(\ldots[X_h, Z]\ldots), Y]\ldots).$$
The second inclusion follows from $f$ being monotone in the $k$-th coordinate and $g_h$ being monotone in the $h$-th coordinate.
As to the third inclusion, by the monotonicity of $g_h$ it is enough to show that $$f(\ldots [g_h(\ldots [X_h,Z]\ldots), Z] \ldots)\subseteq X_h,$$
which follows by applying residuation to the the ``tautological'' inclusion
$$g_h(\ldots [X_h,Z]\ldots)\subseteq g_h(\ldots [X_h,Z]\ldots).$$
}

2. Let us fix $\y\subseteq \p(W)$, $X_1,\ldots,X_{h-1}, X_{h+1}, \ldots X_j\in \p(W)$ and
us show that $$f(X_1,\ldots,\bigcup \y,\ldots, X_j) = \bigcup \{f(X_1,\ldots,Y,\ldots, X_j)\ |\ Y\in \y\}.$$
The right-to-left inclusion follows by $f$ being order-preserving in each coordinate. By residuation, the converse inclusion is equivalent to
$$\bigcup \y\subseteq
g_h(X_1,\ldots,\bigcup \{f(X_1,\ldots,Y,\ldots, X_j)\ |\ Y\in \y\}, \ldots, X_j),$$
to prove which, the following chain suffices:
\begin{center}
\begin{tabular}{r c l}
$\bigcup \y$&
$\subseteq$& $\bigcup\{g_h(X_1,\ldots,f(X_1,\ldots,Y,\ldots, X_j),\ldots, X_j)\ |\ Y\in \y\}$\\
&$\subseteq$&$g_h(X_1,\ldots,\bigcup \{f(X_1,\ldots,Y,\ldots, X_j)\ |\ Y\in \y\}, \ldots, X_j).$\\
\end{tabular}
\end{center}
The first inclusion readily follows by applying residuation to the  ``tautological'' inclusions
$$f(X_1,\ldots,Y,\ldots, X_j) \subseteq f(X_1,\ldots,Y,\ldots, X_j) $$
for every $Y\in \y$.
The second one follows by the monotonicity of $g_h$ in its $h$-th coordinate.

3. Let $\s\subseteq W^{j+1}$ be defined as follows: for all $y, x_1,\ldots, x_j\in W$,
$$(x_1,\ldots, x_j, y)\in \s\ \mbox{ iff }\ \{y\}\subseteq f(\{x_1\}, \ldots \{x_j\}).$$
Then $f(\{x_1\}, \ldots \{x_j\}) = \s[\{x_1\}, \ldots \{x_j\}]$. Using this observation and the fact that, by item 2 above, every residuated map is completely join-preserving in each argument, it is easy to show that for all $X_1,\ldots, X_j\in \p(W)$,
$$f(X_1,\ldots, X_j) = \s[X_1,\ldots, X_j].$$
\end{proof}
%
Consider the following conditions on $\phi$: There exist $k_1, \ldots, k_{n'}\in \mathbb{N}$, all bound above by the modal depth of $\phi$, such that
\begin{enumerate}
\item[(a)]$\phi(p_1, \ldots, p_n) = \phi'(\chi_1({\bm \rho}_{1}, p_{j_1})/q_1, \ldots, \chi_{n'}({\bm \rho}_{n'}, p_{j_{n'}})/q_{n'}, \gamma_1, \ldots, \gamma_{\ell})$, where each placeholder variable $q_1,\ldots, q_{n'}$ occurs exactly once in $\phi'$, and ${\bm \rho}_{i}\in \mathsf{SubSeq}(j_i)$ for each $1\leq i\leq n'$;\footnote{If $j_i = 1$, then ${\bm \rho}_1 = {\bm \epsilon}$, and hence $\chi_i({\bm \rho}_{i}, p_{j_i}) = \chi_i(p_{j_i}) = \chi_i(p_{1})$.}
\item[(b)] for any $\ell$-tuple  of parameters $\overline{A}$, the meaning function $\val{\varphi'(\overline{q}, \overline{A})}$ is a
$\overline{1}$-additive map; 
\item[(c)] for each $1\leq i\leq n'$, 
the meaning function $\val{\chi_i({\bm \rho}_{i}, p_{j_i})}$ is the $(h_i +1)$-th residual of some  residuated operation $f_i$, where $h_i$ is the length of ${\bm \rho}_{i}$;\footnote{This notation signifies that $\chi_i$ and $f_i$  are residuated in their last coordinates.}
%
%
\item[(d)] for every $1\leq i\leq n'$, the map $f_i$ is given by
    \[
    f_i(X_0, \ldots, X_{h_i - 1}, X_{h_i}) = \s_i[X_{h_i}, X_0, \ldots, X_{h_i - 1}]
    \]
for $\s_i$ an $(h_i + 2)$-ary relation given as follows:  for all $w_0,\ldots, w_{h_i}, v \in W$,
    \[
    \s_i(w_0,\ldots, w_{h_i}, v)\ \mbox{ iff }\  (\bigwedge_{0 \leq j < h_i} Rw_j w_{j+1}) \wedge R^{k_i}w_{h_i}  v.
    \]
    The number $k_i$ will be referred to as the \emph{$R$-number of $f_i$}.
\item[(e)] $\val{\gamma_1},\ldots,\val{\gamma_{\ell}}$ are order-reversing in each coordinate.
%
\end{enumerate}

Notice that condition (c) on the $\chi$-formulas in the previous subsection is exactly  condition (c) above restricted to unary maps. In  condition (d) the last coordinate of the map $f_i$ becomes the first coordinate of the the direct image of the relation $\s_i$. This is necessitated by the need to respect two established notational conventions. Firstly, the convention governing the order of coordinates in residuals as described above and, secondly, the convention regarding the order of coordinates in direct images of relations.


For every frame $\f$ and  formula $\phi(p_1, \ldots, p_n)$ satisfying conditions (a)-(e), let $\mathsf{Val}_{3}$ be the set of valuations $V_3$ on $\f$ which map any $q\in \Prop \setminus \{ p_1, \ldots, p_n \}$ to $\emptyset$ and are defined inductively on $\{p_1, \ldots, p_n\}$ as follows:

\begin{enumerate}
\item $V_3(p_1) = \bigcup_{(z,\bm{\epsilon}, k) \in Q_1} \s_{{\bm \epsilon}}^{k} [z]$, where $Q_1\subseteq_m W\times \mathsf{SubSeq(0)}\times \mathbb{N}$ and $m$ is the maximum number of times that a variable occurs in $\phi$, and moreover  $\s_{{\bm \epsilon}}^{k} = R^{k}$ for some $k\in \mathbb{N}$ uniformly bounded by the modal depth of $\phi$; 
\item for every $1 < i \leq n$, $V_3(p_i) = \bigcup_{(z,\bm{\rho}, k) \in Q_i} \s_{{\bm \rho}}^{k}[z, V_3(\bm{\rho}^1),\ldots, V_3(\bm{\rho}^{\ell(\bm{\rho})})]$, where $Q_i\subseteq_m W\times \mathsf{SubSeq(i-1)}\times \mathbb{N}$, and $m$ is the maximum number of times that a variable occurs in $\phi$, and moreover $\s_{\bm{\rho}}^{k}$ is defined for some $k\in \mathbb{N}$ uniformly bounded by the modal depth of $\phi$ as in item (d) above, i.e.\ for all $x_0,\ldots, x_{\ell(\bm{\rho})}, y\in W$,

\begin{equation}
\label{eq:L_0 definable i+1 ary rel}
\s_{\bm{\rho}}^{k}(x_0,\ldots, x_{\ell(\bm{\rho})}, y)\
\mbox{ iff }\  (\bigwedge_{0 \leq h < \ell(\bm{\rho})} Rx_h x_{h+1}) \wedge R^{k}x_{\ell(\bm{\rho})}  y.
\end{equation}
\end{enumerate}

\begin{prop}
\label{linear ind}Let $\varphi\to \psi\in \ML$ be such that $\varphi$
verifies the conditions (a)-(e) above and $\val{\psi}$ is order-preserving in each coordinate. Then the following are equivalent:
\begin{enumerate}
\item  $(\forall V\in \mathsf{Val}(\f))[w\in \val{\varphi}(V) \Rightarrow w\in \val{\psi}(V)]$
\item  $(\forall V_3\in \mathsf{Val}_{3}(\f))[w\in \val{\varphi}(V_3) \Rightarrow w\in \val{\psi}(V_3)]$.
\end{enumerate}
\end{prop}

\begin{proof}
(1 $\Rightarrow$ 2) Clear. (2 $\Rightarrow$ 1) Let  $V \in \mathsf{Val}(\f)$ and $w\in W$ s.t.\  $w \in \val{\varphi}(V)$. Hence, $$\emptyset \neq  \val{\varphi}(V) =  \val{\phi'}(\val{\chi_1}(V), \ldots, \val{\chi_{n'}}(V), \val{\gamma_1}(V), \ldots, \val{\gamma_{\ell}}(V)).$$ By assumption (b) and  Lemma \ref{Op:Composition:Props:Lemma}, this implies that
$\val{\chi_i}(V)\neq \emptyset$ for every  $1\leq i\leq n'$, and
moreover, 
the following chain of equalities holds (in what follows, for any ${\bm \rho}$ we let $V({\bm \rho})$ denote the tuple  $(V({\bm \rho}^{1}),\ldots,V({\bm \rho}^{\ell(\rho)})$):
\begin{center}
\begin{tabular}{r c l}
$w$ &$\in$ &$\val{\phi'}(\val{\chi_1}(V), \ldots, \val{\chi_{n'}}(V), \val{\gamma_1}(V), \ldots, \val{\gamma_{\ell}}(V))$\\
&$=$ &$\bigcup \{ \val{\phi'}(\{x_1\}, \ldots, \{x_{n'}\}, \val{\gamma_1}(V), \ldots, \val{\gamma_{\ell}}(V)) \mid [x_i \in \val{\chi_i}(V)]_{i = 1}^{n'}\}$\\
&$=$ &$\bigcup \{ \val{\phi'}(\{x_1\}, \ldots, \{x_{n'}\}, \val{\gamma_1}(V), \ldots, \val{\gamma_{\ell}}(V)) \mid [f_i(V({\bm \rho}_i), \{x_i\}) \subseteq V(p_{j_i})]_{i = 1}^{n'} \}$\\
\end{tabular}
\end{center}
where the first equality follows from assumption (b) and the second from assumption (c).
Hence, $w \in \val{\phi'}(\{z_1\}, \ldots, \{z_{n'}\}, \val{\gamma_1}(V), \ldots, \val{\gamma_{\ell}}(V))$ for some $z_1,\ldots z_{n'}\in W$ such that
\begin{equation}\label{eq:1}
f_i(V({\bm \rho}_i), \{z_i\}) \subseteq V(p_{j_i})\quad \mbox{ for each } 1 \leq i \leq n'.
\end{equation}

Let $V_3$ be the valuation that maps any $q\in \Prop\setminus\{p_1,\ldots, p_n\}$ to $\emptyset$ and is defined inductively on $\{p_1,\ldots,p_n\}$ as follows:
\begin{itemize}
\item $V_3(p_1) = \bigcup\{ f_j(\{z_j\})\mid f_j(\{z_j\})\subseteq V(p_1)\}$, and
\item for $1<i\leq n$,  $ V_3(p_i)= \bigcup\{ f_j(V_3({\bm \rho}_j), \{z_{j}\})\mid f_j(V({\bm \rho}_j), \{z_j\})\subseteq V(p_i)\}$.
\end{itemize}
We set $Q_1 = \{ (z_j, {\bm \epsilon}, k_j) \mid f_j(\{z_j\})\subseteq V(p_1) \mbox{ and $k_j$ the $R$-number of $f_j$} \}$ and, for $1<i\leq n$, we set $Q_i = \{ (z_j, {\bm \rho}_j, k_j) \mid f_j(V({\bm \rho}_j), \{z_j\})\subseteq V(p_i) \mbox{ and $k_j$ the $R$-number of $f_j$} \}$. By assumption (d),
\[
f_j(V_3({\bm \rho}_j), \{z_{j}\}) = \s_j[z_j, V_3(\bm{\rho}_j^1),\ldots, V_3(\bm{\rho}_j^{\ell(\bm{\rho}_j)})],\]
and hence  $V_3 \in \mathsf{Val}_{3}$. Moreover, by definition  $V_3(p) \subseteq V(p)$, for all $p\in \mathsf{Prop}$.
Also,
\begin{equation}\label{eq:2}
f_i(V_3({\bm \rho_i}),\{z_i\})\subseteq V_3(p_{j_i})\quad \mbox{ for each } 1\leq i\leq n'.\end{equation}
To see this, by definition of $V_3(p_{j_i})$, it is enough to show that $f_i(V({\bm \rho_i}),\{z_i\})\subseteq V(p_{j_i})$, which is true by \eqref{eq:1}.
Let us show that $w \in \val{\varphi}(V_3)$:
\begin{center}
\begin{tabular}{r  l l}
 $w\in $& $\val{\varphi'}(\{z_1\},\ldots, \{z_{n'}\}, \val{\gamma_1}(V), \ldots, \val{\gamma_{\ell}}(V))$\\
 $\subseteq $& $\bigcup\{\val{\varphi'}(\{x_1\},\ldots, \{x_{n'}\}, \val{\gamma_1}(V), \ldots, \val{\gamma_{\ell}}(V))\ |\
[f_i(V_3({\bm \rho_i}),\{x_i\})\subseteq V_3(p_{j_i})]_{i=1}^{n'}\}$ & \eqref{eq:2}\\
 $\subseteq $& $\bigcup\{\val{\varphi'}(\{x_1\},\ldots, \{x_{n'}\}, \val{\gamma_1}(V_3), \ldots, \val{\gamma_{\ell}}(V_3))\ |\
[f_i(V_3({\bm \rho_i}),\{x_i\})\subseteq V_3(p_{j_i})]_{i=1}^{n'}\}$ & assumption (e)\\
 $= $& $\bigcup\{\val{\varphi'}(\{x_1\},\ldots, \{x_{n'}\}, \val{\gamma_1}(V_3), \ldots, \val{\gamma_{\ell}}(V_3))\ |\
[x_i\in \val{\chi_i}(V_3)]_{i=1}^{n'}\}$ &  assumption (c)\\

%
%
 $=$ & $ \val{\varphi'}(\val{\chi_1}(V_3),\ldots, \val{\chi_{n'}}(V_3), \val{\gamma_1}(V_3), \ldots, \val{\gamma_{\ell}}(V_3))$ & assumption (b)\\
 $=$ & $ \val{\varphi}(V_3)$.
\end{tabular}
\end{center}
%


By assumption (2), we can conclude   that $w\in \val{\psi}(V_3)$.  Since  $\val{\psi}$ is order-preserving in each coordinate and  $V_3(p) \subseteq V(p)$ for all $p\in \mathsf{Prop}$, this implies  that  $w\in \val{\psi}(V)$, as required.
\end{proof}
\paragraph{Syntactic conditions.}

\begin{prop}
Any definite atomic inductive implication $\varphi\to \psi\in \ML$  
verifies the assumptions of Proposition \ref{linear ind}. In particular, the maps $\val{\chi_i}$ are exactly those induced by the atomic box-formulas. 
\end{prop}
\begin{proof}
As far as conditions (a) and (b) are concerned, we note the following. If $\varphi$ is a definite atomic inductive antecedent, then the atomic box-formulas and negative formulas used in its construction correspond to   $\chi_1, \ldots \chi_{n'}$ and $\gamma_1, \ldots, \gamma_{\ell}$, respectively, in
\[
\phi'(\chi_1({\bm \rho}_{1}, p_{j_1})/q_1, \ldots, \chi_{n'}({\bm \rho}_{n'}, p_{j_{n'}})/q_{n'}, \gamma_1, \ldots, \gamma_{\ell})
\]
The `skeleton' consisting of the composition of $\Diamond$s and $\wedge$s used in the construction of $\phi$ corresponds to the $\overline{1}$-additive map $\val{\phi'}$ by Lemma \ref{very simple->complete operator}.

As for conditions (c), (d) and (e),  since the dependency digraph of $\varphi$ is acyclic, its transitive closure is a strict partial order. We can therefore assume, without loss of generality, that the variables are ordered  by some linear extension of the partial order $p_1 < p_2 < \cdots < p_n$. Hence we will have that ${\bm \rho}_i \in \mathsf{SubSeq}(j_i-1)$ in each $\val{\chi_i}({\bm \rho}_{i}, p_{j_i})$.  By Lemma \ref{very simple->complete operator},  and Propositions \ref{prop:right adjoints} and \ref{prop:simple sahlqvist syntactic cond}, in order to complete the proof, it is enough to show that for every atomic box-formula
\[
\chi_i({\bm \rho}_i, p_{j_i}) = \Box({\bm \rho}_i^1 \rightarrow \Box({\bm \rho}_i^2\rightarrow \ldots
\Box({\bm \rho}_i^{h_i} \rightarrow \Box^k p_{j_i})\ldots)),
\]
where $h_i$ is the length of ${\bm \rho}_i$, the associated meaning function $\val{\chi_i}$ is the $(h_i + 1)$-th residual of the map $f_i:\p(W)^{h_i + 1}\to \p(W)$ defined as
\[
f(X_{1},\ldots, X_{h_i}, Y) = R^{k_{i}}[X_{h_i}\cap R[ X_{h_i-1}\cap \cdots R[X_1\cap R[Y]\cdots]]].
\]

By induction on the number of inessential variables in $\chi_i$. If there are no inessential variables, then  $\chi_i$ is a boxed atom $\Box^k p$, and the proof is analogous to that of Proposition \ref{prop:simple sahlqvist syntactic cond}. As to the induction step,
let $\chi_i(p, \overline{q}, r) = \Box(p\rightarrow \chi'(\overline{q}, r))$, and let $f'$ be the residual of $\val{\chi'}(\overline{q}, r)$ in the last coordinate. Then for all $X, U, Y$ and $\overline{Z}$,
\[X\subseteq \val{\chi_i}(U, \overline{Z}, Y)\quad \mbox{ iff }\quad X\subseteq l_{R}(U\Rightarrow \val{\chi'}(\overline{Z}, Y)) \quad\mbox{ iff } \quad f'(\overline{Z}, U\cap m_{R^{-1}}(X))\subseteq Y,\]
where $S\Rightarrow V: = -_W S\cup V$.
Then \[f_i(U, \overline{Z}, X)= f'(\overline{Z}, U\cap m_{R^{-1}}(X)) = f'(\overline{Z}, U\cap R[X]),\] and the statement follows by applying the induction hypothesis  to $f'$.
\end{proof}

\begin{example}\label{Atom:Ind:Example}
Let us consider the atomic inductive formula
\[
\phi: = p \wedge \Box(p \to q)  \to \Diamond  q,
\]
which locally corresponds to the property of being a reflexive state. The dependency digraph induces the order $p < q$ on the variables. We have $k_1 = 0$, $m_1 = 0$, $k_2 = 0$ and $m_2 = 1$. The standard local second-order translation is
\[
\forall P \forall Q [P(x) \wedge \forall y (Rxy \rightarrow (P(y) \rightarrow Q(y))) \rightarrow \exists u (Rxu \wedge Q(u))].
\]
The reduction strategy prescribes that we replace $\forall P \forall Q$ in the prefix with $\forall z_1 \forall z_2$ and that we substitute occurrences of the form $P(y)$ with $\alpha_1(y) := \exists u_1(z_1 = u_1 \wedge y = u_1)$ which is equivalent to $y = z_1$. It further prescribes that occurrences of the form $Q(y)$ should be substituted with $\alpha_2(y) := \exists v_0 \exists v_1 (z_2 = v_0 \wedge Rv_0v_1 \wedge v_1 = z_1 \wedge v_1  = y)$, where we have already used the simplified version of $\alpha_1$. Now $\alpha_2(y)$ can be further simplified to $Rz_2 z_1 \wedge z_1  = y$.
Doing the substitution we obtain
\[
\forall z_1 \forall z_2 [x = z_1 \wedge \forall y (Rxy \rightarrow (y = z_1 \rightarrow Rz_2  z_1 \wedge y = z_1 )) \rightarrow \exists u (Rxu \wedge Rz_2  z_1 \wedge u = z_1)].
\]
This is equivalent to
\begin{eqnarray}
&&\forall z_1 \forall z_2 [\forall y (Rxy \rightarrow (y = x \rightarrow  Rz_2  x)) \rightarrow \exists u (Rxu \wedge Rz_2  x \wedge u = x)]\nonumber\\
&\equiv &\forall z_1 \forall z_2 [(Rxx \rightarrow Rz_2  x) \rightarrow \exists u (Rxx \wedge Rz_2  x)]\nonumber\\
&\equiv &Rxx.\nonumber
\end{eqnarray}
\end{example}

\begin{example}
If atomic inductive formulas seem somewhat esoteric, it might be worth noting that there is at least one hiding in plain sight. Indeed, when we rewrite the K axiom $\Box(p\rightarrow q) \rightarrow (\Box p \rightarrow \Box q)$  as $(\Box(p\rightarrow q)\wedge \Box p)\rightarrow \Box q$, we recognize it as an atomic inductive formula where $p <_{\Omega} q$. Since it is valid on all Kripke frames, applying the reduction strategy of course simply produces a first-order validity. Interestingly, there are  relational semantic frameworks even more general than Kripke frames (see e.g.~\cite{Kr65}) in which the K axiom can be shown to still correspond to a validity via a Sahlqvist reduction procedure (cf.\ \cite[Section 8.2]{PaSoZh14a}).

Inductive formulas also pop up in non-classical settings like intuitionistic (modal logic). The Frege axiom $(p\to (q\to r)) \to ((p\to q)\to (p\to r))$ is an inductive formula, although not atomic inductive. In the relational semantics for intuitionistic logic given by pre-orders or posets the Frege Axiom correspondence to a first-order validity. However, when we intepret intuitionistic logic in ternary frames (see e.g.\ \cite{routley-meyer}), the Frege axiom and is no longer valid and corresponds to an informative first-oder condition on these frames.

\end{example}

\section{A glimpse at what lays beyond}\label{sec:glimpse}

In the previous section, we have seen that the order-theoretic properties of $\overline{m}$-additivity, adjunction and residuation form the driving engine of the correspondence phenomenon. This puts us in a position to make a meaningful connection with the general theoretical framework, known as {\em unified correspondence}, simultaneously accounting  for correspondence and canonicity results for wide classes of non-classical logics. The main tools of unified correspondence are: (a) a general definition of Sahlqvist and inductive terms or inequalities which applies uniformly across logical signatures; (b) the calculus for correspondence ALBA, computing the first-order correspondents of input formulas and inequalities in any signature, and guaranteed to be successful on inductive formulas/inequalities in any signature.

As to (a), the definition of `Sahlqvist shape' in any logical signature is given in terms of the order-theoretic properties of the algebraic interpretations of the logical connectives of the given signature. The possibility of providing such a definition is precisely due to the order-theoretic insights that we have illustrated in the previous section, which, in fact, in our presentation are formulated {\em independently} of the specific signature of $\ML$.

As to (b),  ALBA is a syntactic environment (consisting of a propositional language expanding the given signature and a set of rules in the style of proof-theoretic calculi) in which it is possible to encode and automate the metatheoretic reasoning illustrated in the proofs of Propositions \ref{prop:very simple}, \ref{semantic on simple sahl} and \ref{linear ind}. The expanded language includes dedicated variables $\nomj, \nomi$ (called {\em nominals}) and $\cnomm, \cnomn$ (called {\em conominals}) respectively ranging over the atoms (i.e.\ the singleton sets) and co-atoms (i.e.\ the complements of singletons) of the complex algebra of any Kripke frame, and also the adjoints and residuals of every connective in the original signature (hence, in the case of $\ML$, the expanded language includes a new diamond-type connective $\Diamondblack$ the algebraic interpretation of which is $m_{R^{-1}}$, the left adjoint of the semantic interpretation of $\Box$). The expanded language contains all the ingredients needed to express tame valuations as {\em term functions} of a propositional language, rather than by means of first-order formulas. We refer to the bibliography mentioned in the introduction for an exhaustive account of this theory. To try and give an impression of how ALBA works, while at the same time showing the relation between ALBA-reductions and e.g.\  the proof of Proposition \ref{linear ind}, let us consider the subclass of atomic inductive implications on two variables of the following shape: 
\[\phi'(\chi_1(p)/q_1, \chi_2(p, q)/q_2)  \to \psi(p, q),\]


\noindent where $q_1$ and $q_2$ occur exactly once in $\phi'(q_1, q_2)$, and $q$ is the head of  the box-formula $\chi_2$. ALBA takes in input the corresponding quantified inequality
\[\forall p\forall q[\phi'(\chi_1(p), \chi_2(p, q))  \leq \psi(p, q)]\]
\noindent and equivalently transforms it into the following quasi-inequality:

\[\forall p\forall q\forall\nomj\forall \nomi\forall \cnomm [(\nomj\leq \chi_1(p)\ \&\ \nomi\leq\chi_2(p, q)\ \&\  \psi(p, q)\leq \cnomm)\Rightarrow \phi'(\nomj, \nomi)\leq \cnomm].\]

\noindent We leave it to the reader to verify as an exercise that this equivalence is sound on complex algebras. The proof makes use of the fact that  every element of a complex algebra (and hence in particular any possible interpretation of $\chi_1(p)$, $\chi_2(p, q)$ and $\psi(p, q)$) is the union of the singletons of its elements and the intersection of the complements of singletons of its non-elements, and the fact that the meaning function associated with $\phi'$ is $1$-additive. In its turn, the quasi-inequality above can be equivalently rewritten as follows:
 \[\forall p\forall q\forall\nomj\forall \nomi\forall \cnomm [(\xi_1(\nomj)\leq p\ \&\ \xi_2(p, \nomi)\leq q\ \&\  \psi(p, q)\leq \cnomm)\Rightarrow \phi'(\nomj, \nomi) \leq \cnomm],\]
\noindent  where $\xi_1$ and $\xi_2$ are formulas in the expanded language, possibly containing the additional connective $\Diamondblack$, which are respectively interpreted as the adjoint and residual in the last coordinate of the meaning functions associated with $\chi_1$ and $\chi_2$ respectively. The formulas  $\xi_1$ and $\xi_2$ can be computed by induction on $\chi_1$ and $\chi_2$. For example, if $\chi_1(p): = \Box^k p$, then $\xi_1(\nomj): = \Diamondblack^k\nomj$, and if $\chi_2(p, q): = \Box(p\rightarrow \Box^kq)$, then $\xi_2(p, \nomj): =\Diamondblack^k( p\wedge\Diamondblack \nomj)$.
 We claim that the displayed quasi-inequality above is equivalent on complex algebras to the following quasi inequality:

 \[\forall q\forall\nomj\forall \nomi\forall \cnomm [\xi_2(\xi_1(\nomj)/p, \nomi)\leq q\ \&\  \psi(\xi_1(\nomj)/p, q)\leq \cnomm)\Rightarrow \phi'(\nomj, \nomi) \leq \cnomm],\]
\noindent  which is in turn equivalent to:
 \[\forall\nomj\forall \nomi\forall \cnomm [\psi(\xi_1(\nomj)/p, \xi_2(\xi_1(\nomj)/p, \nomi)/q)\leq \cnomm)\Rightarrow \phi'(\nomj, \nomi) \leq \cnomm],\]

\noindent from which all propositional variables have been eliminated, and which can hence be translated inductively into a first-order sentence. Both equivalences are instances of  {\em Ackermann's lemma}, and we refer the reader to \cite{CoGhPa13} for proof and an expanded discussion. Finally, notice that the interpretations on complex algebras of the terms $\xi_1(\nomj)$ and $\xi_2(\xi_1(\nomj)/p, \nomi)$ coincide with  the definition of the assignments of $p = p_1$ and $q = p_2$ under the tame valuation $V_3$ defined in the proof of Proposition \ref{linear ind}.
\section{Conclusions}

In this paper, the Sahlqvist-style syntactic identification of classes of modal formulas that are endowed with  local first-order correspondents has been explained in terms of certain order-theoretic properties of the extension maps corresponding to the formulas of these classes. These properties and the resulting methodology apply also beyond the Sahlqvist class, as the example of atomic inductive formulas shows. Further features which we would like to emphasize are:

\paragraph{Generalizing the signatures.} Our treatment is modular: in particular, we neatly divided the correspondence proof for each class of formulas in three stages. Although, for simplicity, we confined our treatment to the basic modal signature, the most important stage---i.e.\ the one referred to as `order-theoretic conditions'---is  intrinsically independent from any algebraic signature. Therefore, it can be applied to any one, and in particular to any modal signature. This  observation  makes it possible to take the order-theoretic behaviour of the algebraic interpretations of formulas  as {\em primary}, and to provide a principled order-theoretic formulation of the Sahlqvist (or inductive) shape for logics algebraically captured by varieties of  lattice expansions (cf.\ \cite{CoGhPa13,CoPa11}), which can be extended also to other logical frameworks such as hybrid logics \cite{ConRob} and mu-calculi \cite{CFPS}.

\paragraph{Modifying the description of tame valuations.} We have showed that at the core of the correspondence mechanism there are special  classes of (tame) valuations, the members of which can be described uniformly in the language $L_0$. Of course the classes $\mathsf{Val}_1(\f)$, $\mathsf{Val}_2(\f)$  and $\mathsf{Val}_3(\f)$, on which we settled as a compromise between simplicity of presentation and generality, are just three instances, and a plethora of further refinements are possible, reaching out to the class of inductive implications, which are characterized by the following shape of box-formulas:
%
%
\[
\Box(A_{0}\rightarrow \Box(A_{1}\rightarrow \ldots \Box(A_{n}\rightarrow \Box^k p)\ldots)),
\]
where the $A_i$s are arbitrary positive formulas. The correspondence result for the whole class of inductive formulas can be obtained by the same methodology we proposed in Section \ref{Atomic:Ind:Section}, applied to a suitably larger class of  tame valuations.

However, as we have seen in Section \ref{sec:glimpse}, the alternative approach of ALBA makes it possible to directly {\em compute} the tame valuations from the given formula in input, rather than having first to target a specific class of tame valuations. Hence, ALBA is computationally a more convenient tool, while at the same time being based on the same order-theoretic principles as the proofs given in Section \ref{section:algebraic correspondence}, and implementing the same `minimal valuation' proof strategy.

\paragraph{From the Boolean to more general settings.}
The statements and proofs of our theorems about the order-theoretic conditions only use the following two features of powerset algebras: that they are complete distributive lattices, and that they are completely join-generated by their completely join prime elements\footnote{An element $c\neq \bot$ of a complete lattice $L$ is \emph{completely join prime} if, for every $S\subseteq L$,  $c \leq s$ for some $s\in S$ whenever $c
\leq \bigvee S$.} (the singleton subsets). Therefore, these proofs go through virtually unchanged in the more general setting of distributive lattices enjoying these two properties. In fact, further refinements are possible, which show that the distributivity of the lattices is also inessential.
This observation motivates the development of the approach to correspondence in which the algebraic perspective is taken as primary (cf.\ e.g.\ \cite{CoGhPa13,CoPa11,CFPS,PaSoZh14a}). This approach turns out to be very fruitful also with respect to the study of several types of relational structures (see e.g.\ the two relational semantics for lattice-based logics discussed in \cite{CoPa11}, and the epistemic interpretation of lattice-based modal logic on RS-frames \cite{CFPPTW}). 

\paragraph{Duality and correspondence.}  Although it was not strictly needed for our exposition, and hence not so prominent in it, the mathematical background of the algebraic approach to correspondence theory is the duality between Kripke frames and complete and atomic BAOs. For instance, results such as  Proposition \ref{prop:right adjoints}.2 and .3  and Proposition \ref{prop:residuated maps}.3 are essentially  characterizations of objects across a duality.
More generally, the results of  the present paper are grounded on the possibility of translating the correspondence problem from  the setting of Kripke frames to that of their associated complex algebras. However, duality guarantees that the converse direction is possible. Namely,  relational structures can be systematically generated from certain algebras. This direction plays a fundamental role in the correspondence theory of substructural logics and lattice-based modal logics.

\bibliographystyle{abbrv}
\bibliography{CPSFinal}

\end{document}